\documentclass{article}
\pdfoutput=1
\usepackage[pdfencoding=auto, psdextra, pagebackref, breaklinks]{hyperref}

\usepackage{L_func}
\usepackage[title]{appendix}

\usepackage{chngcntr}
\counterwithin{section}{part}
\usepackage[titles]{tocloft}

\title{Global Parabolic Induction and Abstract Automorphicity}
\author{Gal Dor\\
Tel-Aviv University}
\date{February 2021}

\begin{document}

\maketitle

\begin{abstract}
    In \cite{abst_aut_reps_arxiv}, the author has constructed a category of \emph{abstractly automorphic representations} for $\GL(2)$ over a function field $F$. This is a symmetric monoidal Abelian category, constructed with the goal of having the irreducible automorphic representations as its simple objects. The goal of this paper is to systematically study this category.
    
    We will prove several structural theorems about this category. We will show that it admits an adjoint pair $(r^\aut,i^\aut)$ of automorphic parabolic restriction and induction functors, respectively. This will allow us to show that the category of abstractly automorphic representations decomposes into cuspidal and Eisenstein components, in analogy with the Bernstein decomposition of the category of $p$-adic representations.
    
    Moreover, along the way, we will give a new perspective on the intertwining operator of $\GL(2)$ (and on the functional equation for Eisenstein series), as a form of self-duality of the functor of parabolic induction. We will also illustrate how the role of analytic continuation in this theory can be thought of as trivializing a twist by a certain line bundle, which corresponds to an L-function via the results of \cite{modules_of_zeta_integrals_arxiv}. If one chooses to keep the twist as a part of the theory, then one avoids the need for analytic continuation.
\end{abstract}

\tableofcontents

\part{Introduction} \label{part:introduction}

Let $\GL_n$ be the general linear group over a function field $F$ of characteristic $\neq 2$. This paper is meant as a step towards answering the question ``what should be considered an automorphic representation of $\GL_n$?''. The usual definition is that an irreducible representation of $\GL_n(\AA)$ is automorphic if and only if it can be given as a subquotient of the space of smooth functions on the automorphic quotient $\GL_n(F)\backslash \GL_n(\AA)$.

In \cite{abst_aut_reps_arxiv}, the author has proposed a definition for the case of $\GL(2)$ which is no longer restricted to irreducible representations. This was done by defining the Abelian symmetric monoidal category $\Mod^\aut(\GL_2(\AA))$, the category of \emph{abstractly automorphic representations}. The category $\Mod^\aut(\GL_2(\AA))$ admits a forgetful functor
\[
    \Mod^\aut(\GL_2(\AA))\ra\Mod(\GL_2(\AA))
\]
to the usual category of smooth $\GL_2(\AA)$-modules. This functor is fully-faithful, and its essential image contains all of the irreducible automorphic (in the usual sense) representations of $\GL_2(\AA)$. The paper \cite{abst_aut_reps_arxiv} proposes that the category of abstractly automorphic representations $\Mod^\aut(\GL_2(\AA))$ is a good setting for the study of automorphic phenomena.

Unfortunately, due to space constraints, the paper \cite{abst_aut_reps_arxiv} did not go into details about the structural theory of $\Mod^\aut(\GL_2(\AA))$. The main goal of this paper is to remedy this situation, by thoroughly investigating the properties of $\Mod^\aut(\GL_2(\AA))$, and showing that it has many of the desirable properties one would want from a ``category of automorphic representations''. Essentially, this paper is a continuation of the research in \cite{abst_aut_reps_arxiv}.

Along the way, we will be able to bring to bear some highly non-trivial categorical tools, and use them to answer questions of automorphic nature, such as questions about intertwining operators and functional equations for Eisenstein series. Hopefully, this will demonstrate how this formalism enables the use of more categorical methods in the study of automorphic representations.

Before giving a more thorough introduction, let us briefly sketch an example for the favorable properties and some of the advantages of the construction $\Mod^\aut(\GL_2(\AA))$.

There is an analogy between the local theory of $p$-adic representations and the global theory of automorphic representations. Both theories admit notions of induction via parabolic subgroups (parabolic induction in the local theory and Eisenstein series in the global theory). The notion of ``parabolic induction'' naturally yields a notion of ``cuspidality'', which is associated with nice growth properties (this notion is called supercuspidality in the local theory, and cuspidality in the global theory).

However, this analogy always occurs in one lower level of categorification in the global case. Parabolic restriction and induction are functors, while constant terms and Eisenstein series are functions.

With the framework of abstractly automorphic representations, this analogy becomes much more precise. That is, it turns out that $\Mod^\aut(\GL_2(\AA))$ decomposes as a category into a cuspidal and an Eisenstein part. Instead of being mere functions, terms such as constant terms and Eisenstein series translate into a pair of adjoint functors, which we refer to as \emph{automorphic parabolic induction and restriction}. Automorphic parabolic restriction annihilates the cuspidal part of the category, and automorphic parabolic induction generates the Eisenstein part. The end result is a decomposition analogous to the Bernstein decomposition in the local case. The only additional subtlety occurs at the so-called ``anomalous spectrum'' (see Warning~\ref{warn:eisenstein_killed_by_rest} below).

The structure of this introduction is as follows. Our goal in this paper is to dissect the structure of the category $\Mod^\aut(\GL_2(\AA))$ in detail. As a result, we will start by giving detailed, clean exposition of its structure, stated without proofs. This will be done in Section~\ref{sect:GL2}, where we will give a summary of the results we intend to prove about $\Mod^\aut(\GL_2(\AA))$. In Section~\ref{sect:GLn}, we speculate on how these ideas might possibly be generalized to other general linear groups. Finally, in Section~\ref{sect:structure_of_paper}, we will give an overview of the structure of the body of this paper.

\section{Summary of Results for \texorpdfstring{$\GL(2)$}{GL(2)}} \label{sect:GL2}

In this subsection, we will summarize the properties of the category of abstractly automorphic representations $\Mod^\aut(\GL_2(\AA))$. This subsection is a list of results, stated without proofs. Some of the proofs have already appeared in \cite{abst_aut_reps_arxiv}, and the rest will appear in the body of this paper.

The goal of this chapter is to cleanly state all of the structural theorems we intend to prove about $\Mod^\aut(\GL_2(\AA))$ in one place. For the sake of organization, we further sub-divide this subsection into topics as follows. After introducing the basic properties of $\Mod^\aut(\GL_2(\AA))$ as a plain category in Subsection~\ref{subsect:basic_properties}, we describe its symmetric monoidal structure in Subsection~\ref{subsect:monoidal}, and its relation with automorphic parabolic induction and restriction (which are related to Eisenstein series) in Subsection~\ref{subsect:parabolic}. In Subsection~\ref{subsect:decomp} we show how one can use these automorphic parabolic induction and restriction functors to describe cuspidal and Eisenstein automorphic representations on the same footing as the local supercuspidal and principle series representations.

\subsection{Basic Properties} \label{subsect:basic_properties}

The first statement we make in this subsection is asserting the existence of a complete and co-complete Abelian category $\Mod^\aut(\GL_2(\AA))$, equipped with a colimit-preserving symmetric monoidal structure denoted by $\oY$.

However, by itself, this statement is fairly empty. One must relate the category $\Mod^\aut(\GL_2(\AA))$ to known categories in order to have a sensible mathematical statement. Over the rest of this subsection, we will describe various constructions and properties that the category $\Mod^\aut(\GL_2(\AA))$ possesses, which will make this statement meaningful.
\begin{theorem}[Claim~4.25, Proposition~4.29 and Remark~4.30 of \cite{abst_aut_reps_arxiv}] \label{thm:realization}
    There is a canonical realization functor
    \[
        \iota\co\Mod^\aut(\GL_2(\AA))\ra\Mod(\GL_2(\AA))
    \]
    to the category of smooth $\GL_2(\AA)$-modules, satisfying the following:
    \begin{enumerate}
        \item The functor $\iota$ is fully faithful.
        \item The functor $\iota$ respects all limits and colimits. In particular, it is exact.
        \item The essential image of $\iota$ is closed under taking subquotients and contragradients.
        \item The space $\cS_G=S(\GL_2(F)\backslash\GL_2(\AA))$ of smooth and compactly supported functions on $\GL_2(F)\backslash\GL_2(\AA)$ lies in the essential image of $\iota$.
    \end{enumerate}
\end{theorem}
The above is still not enough to uniquely specify $\Mod^\aut(\GL_2(\AA))$ and $\iota$. However, the specific construction enjoys many more properties, which we describe below.

\begin{warning}
    Note that despite the fact that the essential image of $\iota$ is closed under subquotients, it is not closed under extensions.
\end{warning}

Observe that Theorem~\ref{thm:realization} implies that $\iota$ takes irreducible objects to irreducible objects, and moreover that any irreducible automorphic representation from $\Mod(\GL_2(\AA))$ lies in the essential image of $\iota$. We will prove that the converse also holds:
\begin{theorem}[Theorem~\ref{thm:abst_aut_irr_is_aut}] \label{thm:intro_abst_aut_irr_is_aut}
    The realization $\iota(M)$ of an irreducible object $M\in\Mod^\aut(\GL_2(\AA))$ is an irreducible automorphic representation, in the sense of being a subquotient of the contragradient $\widetilde{\cS_G}$ of $\cS_G=S(\GL_2(F)\backslash\GL_2(\AA))$.
\end{theorem}

In particular, the functor $\iota$ induces a bijection between irreducible objects of $\Mod^\aut(\GL_2(\AA))$ and irreducible automorphic representations in $\Mod(\GL_2(\AA))$ in the classical sense.

Recall that the category $\Mod(\GL_2(\AA))$ has a large center $Z$ acting on it.
\begin{claim}[Follows from Lemma~2.14 of \cite{abst_aut_reps_arxiv}] \label{claim:cent}
    The center $Z$ has a unique action on $\Mod^\aut(\GL_2(\AA))$ which is compatible with the realization functor
    \[
        \iota\co\Mod^\aut(\GL_2(\AA))\ra\Mod(\GL_2(\AA)).
    \]
\end{claim}
In particular, objects in $\Mod^\aut(\GL_2(\AA))$ of finite length have decompositions into generalized Hecke eigenspaces.

\begin{remark}
    The following claim is beyond the scope of this paper, and will not be proven here. Observe that the category $\Mod^\aut(\GL_2(\AA))$ is enriched over $\Vect=\Vect_\CC$, the category of complex vector spaces. However, it is possible to show that it actually has a \emph{rational structure}. This means that $\Mod^\aut(\GL_2(\AA))$ is given by base change from a category $\Mod_\QQ^\aut(\GL_2(\AA))$ enriched over the category of rational vector spaces $\Vect_\QQ$:
    \[
        \Mod^\aut(\GL_2(\AA))=\Vect_\CC\otimes_{\Vect_\QQ}\Mod_\QQ^\aut(\GL_2(\AA)).
    \]
\end{remark}

\subsection{Symmetric Monoidal Structure} \label{subsect:monoidal}

As mentioned above, the category $\Mod^\aut(\GL_2(\AA))$ of abstractly automorphic representations carries a symmetric monoidal structure, which we denote by $\oY$. This notation was chosen deliberately in order to emphasize that this monoidal structure is \emph{not} compatible with any standard symmetric monoidal structure of $\Mod(\GL_2(\AA))$. In this subsection, we describe some of the properties of this symmetric monoidal structure.

First, we claim that:
\begin{claim}[Lemma~2.14 of \cite{abst_aut_reps_arxiv}] \label{claim:cent_monoidal}
    The symmetric monoidal structure $\oY$ is compatible with the action of the center $Z$ on $\Mod^\aut(\GL_2(\AA))$.
\end{claim}
In fact, we can also make the following statement. Recall that an irreducible automorphic representation $V\in\Mod(\GL_2(\AA))$ is called \emph{generic} if it admits a Whittaker model.
\begin{claim}[Example~3.54 of \cite{abst_aut_reps_arxiv}] \label{claim:idem}
    Let $M\in\Mod^\aut(\GL_2(\AA))$ be an irreducible object. Then the realization $\iota(M)$ is a generic automorphic representation if and only if $M$ admits a non-zero map from the unit of the symmetric monoidal structure $\oY$. Moreover, such a map induces an isomorphism:
    \[
        M\cong M\oY M.
    \]
\end{claim}

\begin{remark}
    Note that Claims~\ref{claim:cent_monoidal} and~\ref{claim:idem} imply that $\oY$ is well-behaved with respect to the spectral decomposition of objects in $\Mod^\aut(\GL_2(\AA))$. This is very different from the behaviour of the usual tensor product $\otimes$ on the category $\Mod(\GL_2(\AA))$, where the product of irreducible objects is typically very far from irreducible.
\end{remark}

Additionally, we can describe the realization of the unit of $\oY$. Recall that $\cS_G=S(\GL_2(F)\backslash\GL_2(\AA))$ is the space of smooth and compactly supported functions on $\GL_2(F)\backslash\GL_2(\AA)$. Let $\cI_G\subseteq\cS_G$ be the subspace consisting of functions that are orthogonal to all characters of the form $\chi(\det(g))$, for $\chi\co\AA^\times/F^\times\ra\CC^\times$ a grossencharacter.

Then we claim that:
\begin{claim}[Follows from Definition~4.26 of \cite{abst_aut_reps_arxiv}] \label{claim:unit}
    The realization of the unit of the symmetric monoidal structure $\oY$ is canonically isomorphic to $\cI_G$.
\end{claim}

\subsection{Automorphic Parabolic Induction and Restriction} \label{subsect:parabolic}

Let $T$ be the torus $\GL_1\times\GL_1$. Denote by $\Mod^\aut(T(\AA))$ the category of smooth representations of the Abelian group $T(\AA)/T(F)$. This category admits a canonical realization functor
\[
    \iota\co\Mod^\aut(T(\AA))\ra\Mod(T(\AA))
\]
into the category of smooth $T(\AA)$-modules, given by restriction along the map
\[
    T(\AA)\ra T(\AA)/T(F).
\]
Moreover, $\Mod^\aut(T(\AA))$ admits a symmetric monoidal structure $\otimes_{T(\AA)/T(F)}$ given by the relative tensor product over $T(\AA)/T(F)$. Let us abuse notation and denote this symmetric monoidal structure by $\oY$ as well.

It is easy to see that $\Mod^\aut(T(\AA))$, along with $\iota$ and $\oY$, satisfies the appropriate analogues of Theorems~\ref{thm:realization} and~\ref{thm:intro_abst_aut_irr_is_aut}, as well as Claims~\ref{claim:cent}, \ref{claim:cent_monoidal}, \ref{claim:idem} and~\ref{claim:unit}.

Therefore, it is natural to ask about the compatibility of $\Mod^\aut(T(\AA))$ with $\Mod^\aut(\GL_2(\AA))$. Indeed, it turns out that the two structures are intimately related.

Fix a parabolic subgroup $P\subseteq\GL_2$, with Levi given by $T$. Let
\[
    i\co\Mod(T(\AA))\ra\Mod(\GL_2(\AA))
\]
be the usual (un-normalized) parabolic induction functor with respect to $P$. Our first claim is that this functor respects abstract automorphicity:
\begin{theorem}[Theorem~\ref{thm:i_hat_rlax_global} and Theorem~\ref{thm:i_aut_to_aut}] \label{thm:parabolic_ind}
    There is an essentially unique functor $i^\aut$ completing the diagram:
    \[\xymatrix{
        \Mod^\aut(T(\AA)) \ar@{-->}[d]^{i^\aut} \ar[r]^\iota & \Mod(T(\AA)) \ar[d]^{i} \\
        \Mod^\aut(\GL_2(\AA)) \ar[r]^\iota & \Mod(\GL_2(\AA)).
    }\]
    The functor $i^\aut$ is non-unital right-lax with respect to the symmetric monoidal structure $\oY$.
\end{theorem}

\begin{recollection} \label{recall:right_lax}
    Recall that a \emph{non-unital right lax} functor between the two monoidal categories $\C$ and $\D$ is a functor $F\co\C\ra\D$ equipped with a natural transformation (referred to as the \emph{multiplication} on $F$):
    \[
        F(A)\otimes_\D F(B)\ra F(A\otimes_\C B),
    \]
    satisfying an appropriate associativity property. Essentially, this means that $F$ sends non-unital algebras to non-unital algebras.
    
    Similarly, there is a notion of a \emph{unital right-lax functor}, which is also equipped with a unit map $\one_\D\ra F(\one_\C)$ satisfying appropriate compatibility conditions. A unital right-lax functor sends unital algebras to unital algebras.
    
    The reader can find out more about these notions in, e.g., \cite{nlab:monoidal_functor}.
\end{recollection}

\begin{remark}
    The non-unital right-lax structure on the functor $i^\aut$ is non-canonical. Rather, it turns out that there is a twist $\widehat{i}^\aut$ of $i^\aut$, which is canonically non-unital right-lax. The twist $\widehat{i}^\aut$ is isomorphic $i^\aut$, but the isomorphism between them is non-canonical.
    
    We will give a more detailed explanation of this phenomenon in Section~\ref{sect:hat_i} below. For now, we can give an informal summary. It turns out that the difference between $i^\aut$ and $\widehat{i}^\aut$ is related to the analytic continuation process usually used in defining the intertwining operator. Specifically, $\widehat{i}^\aut$ is a twist of $i^\aut$ by a (non-canonically) trivial line bundle, which has a canonical trivialization after analytic continuation.
    
    In particular, $\widehat{i}^\aut$ enjoys a canonical description of the intertwining operator which does not require analytic continuation. That is, the role of the analytic continuation in the usual theory is merely to trivialize the twist giving $\widehat{i}^\aut$ from $i^\aut$.
    
    This issue is also related to the L-function that appears in the formula for the functional equation for Eisenstein series. Specifically, it turns out that the space by which the functor $i^\aut$ needs to be twisted is precisely the space of zeta integrals for this L-function, in the sense of \cite{modules_of_zeta_integrals_arxiv}. See also Remark~\ref{remark:L_is_L_func}.
\end{remark}

\begin{remark}
    One can also obtain a treatment of the intertwining operator that does not require analytic continuation by taking Fourier transforms (essentially, Kirillov models). See, e.g., Chapter~2, Section~3, Subsection~3 of \cite{generalized_functions_aut}.
\end{remark}

\begin{remark}
    In the course of proving Theorem~\ref{thm:parabolic_ind}, we will also obtain the functional equation for Eisenstein series as an almost immediate corollary. This relative simplicity is a benefit of our approach. See Remark~\ref{remark:eisenstein_functional_eqs}.
\end{remark}

We refer to $i^\aut$ as \emph{automorphic parabolic induction}. Moreover, it turns out that while the usual parabolic restriction functor
\[
    r\co\Mod(\GL_2(\AA))\ra\Mod(T(\AA))
\]
is very badly behaved with respect to automorphic representations, one can still have a good notion of automorphic parabolic restriction:
\begin{theorem}[Remark~\ref{remark:aut_par_rest} and Theorem~\ref{thm:fin_len_rest}]
    The functor
    \[
        i^\aut\co\Mod^\aut(T(\AA))\ra\Mod^\aut(\GL_2(\AA))
    \]
    has a left adjoint $r^\aut$, referred to as \emph{automorphic parabolic restriction}.
    
    Moreover, the functor $r^\aut$ sends objects of finite length to objects of finite length.
\end{theorem}

\begin{remark}
    Note that unlike $r^\aut$, the usual parabolic restriction functor $r$ does not send objects of finite length to objects of finite length. For example, an irreducible automorphic representation may be principle series at every place, in which case its ``usual'' parabolic restriction has infinite length.
\end{remark}

It is natural to expect that the functor $r^\aut$ kill abstractly automorphic representations that are ``cuspidal''. Indeed:
\begin{claim}[Corollary~\ref{cor:aut_par_rest_kills_cusp}] \label{claim:aut_rest_kills_cusp}
    If $M\in\Mod^\aut(\GL_2(\AA))$ is irreducible, and its realization $\iota(M)$ is a cuspidal automorphic representation, then $r^\aut(M)=0$.
\end{claim}

\begin{remark}
    Note that the analogue of Claim~\ref{claim:aut_rest_kills_cusp} fails to hold for the usual parabolic restriction functor $r$. Indeed, there are cuspidal irreducible automorphic representations which are principle series at every place, and thus their usual parabolic restriction is non-zero. This illustrates the favorable properties of the automorphic parabolic restriction $r^\aut$.
\end{remark}

\begin{warning} \label{warn:eisenstein_killed_by_rest}
    While the functor $r^\aut$ kills all abstractly automorphic representations whose realization is cuspidal, the converse is not true. The problem is that there are irreducible automorphic representations that appear as subquotients in parabolic induction from automorphic characters, but never as sub-representations. Abstractly automorphic representations with such realizations are killed by $r^\aut$, despite clearly being ``Eisenstein'' in nature.
    
    These examples occur in the so-called \emph{anomalous}, or \emph{non-isobaric}, spectrum of $\GL(2)$. Specifically, if we let $\triv_T\in\Mod^\aut(T(\AA))$ be the trivial $T(\AA)/T(F)$-module, then the induction:
    \[
        V=i^\aut(\triv_T)
    \]
    is infinitely reducible. The object $V$ contains a unique irreducible sub-object, whose realization is the trivial $G$-module $\triv_G$. The remaining irreducible subquotients have realizations that are trivial in almost all places, but Steinberg in the rest. It is possible to show that the quotient of $V$ by its unique irreducible sub-object is killed by the automorphic parabolic restriction functor $r^\aut$.
\end{warning}

\begin{warning} \label{warn:supercuspidal_vs_cuspidal_notation}
    Warning~\ref{warn:eisenstein_killed_by_rest} causes an unfortunate clash of notation around the term ``cuspidal''.
    
    The situation mentioned in Warning~\ref{warn:eisenstein_killed_by_rest}, where some representations are subquotients of parabolically induced ones but are still killed by parabolic restriction, is not unique to the abstractly automorphic setting. For example, when one considers the representations of $p$-adic groups with coefficients in $\overline{\FF_\ell}$, $\ell\neq p$, a similar situation occurs.
    
    In such settings, it is customary to use the term ``cuspidal'' to refer to any representation killed by parabolic restriction, and reserve the term ``supercuspidal'' to refer to representations that have nothing at all to do with parabolic induction. See \cite{classification_of_cuspidal_reps_mod_ell} for a randomly chosen modern example, and compare the language of \cite{reps_for_gl_2_mod_ell} and \cite{langlands_for_gl_n_mod_ell}. Unfortunately, it is not customary to refer to anomalous representations, as in Warning~\ref{warn:eisenstein_killed_by_rest}, as ``cuspidal''. Moreover, it is also customary for cuspidal automorphic representations to not be referred to as ``supercuspidal''.
    
    This paper is aimed at primarily at researchers of automorphic representation theory. Therefore, the author has chosen to adopt the convention of this theory over its more esoteric counterparts. In particular, in this text, the term ``cuspidal'' is used to refer to representations that are wholly unrelated to parabolic induction, while the term ``anomalous'' is used to refer to representations killed by parabolic restriction yet are not cuspidal.
\end{warning}

\subsection{Bernstein Decomposition} \label{subsect:decomp}

In this subsection, we will describe a few additional properties of the category $\Mod^\aut(\GL_2(\AA))$. Specifically, we will discuss its decomposition into components, in a way analogous to the local Bernstein decomposition.

In the local theory of smooth $\GL_2$-modules over local fields, one decomposes the category of smooth $\GL_2$-modules into components. Each component is either supercuspidal, or induced from pairs of characters of $\GL_1$. This decomposition is done via the parabolic induction and restriction functors $i,r$, and supercuspidal components are killed by $r$. This is called the Bernstein decomposition.

It turns out that a similar decomposition holds for the global category $\Mod^\aut(\GL_2(\AA))$. Let
\[
    C(\GL_2(\AA))=C^\Eis(\GL_2(\AA))\coprod C^\cusp(\GL_2(\AA))
\]
be the disjoint union of the set $C^\Eis(\GL_2(\AA))$ with the set $C^\cusp(\GL_2(\AA))$. Here, $C^\Eis(\GL_2(\AA))$ is the set of unordered pairs of characters $\AA^\times/F^\times\ra\CC^\times$ up to continuous twists $\abs{\cdot}^{s_1}\otimes\abs{\cdot}^{s_2}$. The set $C^\cusp(\GL_2(\AA))$ consists of irreducible cuspidal automorphic representations up to continuous twist $\abs{\det(\cdot)}^s$. Using the functors $i^\aut,r^\aut$ from Subsection~\ref{subsect:parabolic} above, we will show:
\begin{theorem}[Theorem~\ref{thm:decom_to_comps}] \label{thm:decomp}
    There is a canonical decomposition
    \[
        \Mod^\aut(\GL_2(\AA))=\prod_{c\in C(\GL_2(\AA))}\Mod^\aut_c(\GL_2(\AA))
    \]
    compatible with the symmetric monoidal product $\oY$.
    
    Moreover, the restriction of $r^\aut$ to $\Mod^\aut_c(\GL_2(\AA))$ with $c\in C^\cusp(\GL_2(\AA))$ is $0$.
\end{theorem}

In fact, note that an irreducible $M\in\Mod^\aut(\GL_2(\AA))$ always lies in a single component $\Mod^\aut_c(\GL_2(\AA))$ for some $c\in C(\GL_2(\AA))$. The decomposition above has the property that if the realization of $M$ is an irreducible cuspidal automorphic representation, then $c\in C^\cusp(\GL_2(\AA))$ is the component corresponding to $\iota(M)$. Likewise, if the realization of $M$ is a subquotient of the Eisenstein series parabolically induced from $\chi_1\otimes\chi_2$, then $c\in C^\Eis(\GL_2(\AA))$ corresponds to $\chi_1\otimes\chi_2$.

\begin{remark} \label{remark:subquot_of_eis_not_quot}
    As mentioned in Warning~\ref{warn:eisenstein_killed_by_rest}, due to the anomalous spectrum, there are some subquotients of parabolic inductions $i^\aut(\chi_1\otimes\chi_2)$ which cannot be expressed as sub-modules. In particular, for these specific representations, we have that $r^\aut$ is equal to $0$ despite the fact that these representations belong to Eisenstein components $c\in C^\Eis(\GL_2(\AA))$.
    
    A similar issue appears in the kind of Bernstein decomposition that occurs in the case of $p$-adic representations with coefficients in characteristic $\ell\neq p$, as in \cite{reps_for_gl_2_mod_ell}.
\end{remark}

\section{Speculation for \texorpdfstring{$\GL(n)$}{GL(n)}} \label{sect:GLn}

The goal of this section is to speculate about possible generalizations of the results of this paper, stated in Section~\ref{sect:GL2}, to other general linear groups $G=\GL_n$, with $n>2$. Other than Remark~\ref{remark:reductive_groups} at the very end of this section, we will not consider more general reductive groups. Unlike Section~\ref{sect:GL2}, nothing here is proven. Instead, this section merely contains speculation based on the author's experience in proving the results of Section~\ref{sect:GL2}. In fact, there might even be non-trivial obstructions for some of the speculation here! Hence, we ask that the reader treat the speculation of this section with the caution appropriate for educated guesses.

The text below is essentially a re-formulation of Section~\ref{sect:GL2}, only stated in greater generality. As a result, we will be much more succinct.

Let $F$ be a function field of characteristic $\neq 2$. Our guess is that for each $G=\GL_n(\AA)$ there should be a complete and co-complete Abelian symmetric monoidal category $\Mod^\aut(G)$, equipped with a colimit-preserving symmetric monoidal structure denoted by $\oY$. Of course, this claim is vacuous by itself. Instead, we conjecture that for every $G=\GL_n(\AA)$ there is such a category, with realization functors and compatibility properties as follows:
\begin{enumerate}
    \item There is a realization functor $\iota\co\Mod^\aut(G)\ra\Mod(G)$ to the category of smooth $G$-modules, satisfying the following:
    \begin{enumerate}
        \item The functor $\iota$ is fully faithful.
        \item The functor $\iota$ respects all limits and colimits. In particular, it is exact.
        \item The essential image of $\iota$ is closed under taking subquotients and contragradients.
        \item The space $\cS_G=S(G(F)\backslash G(\AA))$ of smooth and compactly supported functions on $G(F)\backslash G(\AA)$ lies in the essential image of $\iota$.
    \end{enumerate}
    Note that this implies that $\iota$ sends irreducible objects to irreducible objects, and that all irreducible automorphic representations lie in the essential image of $\iota$.
    
    \item The realization $\iota(M)$ of an irreducible object $M\in\Mod^\aut(\GL_n(\AA))$ is an irreducible automorphic representation, in the sense of being a subquotient of the contragradient $\widetilde{\cS_G}$ of $\cS_G$.

    \item Let $Z(G)$ denote the center of the category $\Mod(G)$. Then $Z(G)$ has a unique action on $\Mod^\aut(G)$ which is compatible with the realization functor $\iota\co\Mod^\aut(G)\ra\Mod(G)$.
    
    \item The symmetric monoidal structure $\oY$ is compatible with the action of the center $Z(G)$ on $\Mod^\aut(G)$.
    
    \item Let $M\in\Mod^\aut(G)$ be an irreducible object. Then the realization $\iota(M)$ is a generic automorphic representation if and only if $M$ admits a non-zero map from the unit of the symmetric monoidal structure $\oY$. Moreover, such a map induces an isomorphism:
    \[
        M\cong M\oY M.
    \]
    
    \item Let $\cI_G$ be the subspace of $\cS_G=S(G(F)\backslash G(\AA))$ which is orthogonal to all automorphic forms of residue of Eisenstein series type. Then the realization of the unit of the symmetric monoidal structure $\oY$ is canonically isomorphic to $\cI_G$.
    
    \item Let $P\subseteq G$ be a parabolic subgroup, and let $M$ be the corresponding Levi. Then $M\cong\prod\GL_{n_i}$. We note that the category of smooth $M$-modules is the Lurie tensor product
    \[
        \Mod(M)=\bigotimes\Mod(\GL_{n_i}(\AA)),
    \]
    and let
    \[
        \Mod^\aut(M)=\bigotimes\Mod^\aut(\GL_{n_i}(\AA))
    \]
    also be the Lurie tensor product. We define the realization functor 
    \[
        \iota\co\Mod^\aut(M)\ra\Mod(M)
    \]
    to be the tensor product of the realization functors for each $\GL_{n_i}$. Let
    \[
        i\co\Mod(M)\ra\Mod(G)
    \]
    be the parabolic induction functor along $P$. Then there is an essentially unique functor $i^\aut$ completing the diagram:
    \[\xymatrix{
        \Mod^\aut(M) \ar@{-->}[d]^{i^\aut} \ar[r]^\iota & \Mod(M) \ar[d]^{i}  \\
        \Mod^\aut(G) \ar[r]^\iota & \Mod(G).
    }\]
    
    \item The functor $i^\aut$ (or one of its twists) is non-unital right-lax with respect to the symmetric monoidal structure $\oY$.
    
    \item The functor $i^\aut\co\Mod^\aut(M)\ra\Mod^\aut(G)$ has a right adjoint $r^\aut$, satisfying the following:
    \begin{enumerate}
        \item The functor $r^\aut$ sends objects of finite length to finite length.
        \item If $M\in\Mod^\aut(G)$ is irreducible, and its realization $\iota(M)$ is a cuspidal automorphic representation, then $r^\aut(M)=0$.
    \end{enumerate}
    
    \item The category $\Mod^\aut(G)$ admits a decomposition
    \[
        \Mod^\aut(G)=\prod_{c\in C(G)}\Mod^\aut_c(G)
    \]
    compatible with the symmetric monoidal product $\oY$, for an appropriate set of components $C(G)$ indexed by cuspidal automorphic representations of $G$ and its subgroups.
\end{enumerate}

\begin{remark} \label{remark:reductive_groups}
    The above speculation is probably \emph{not} quite true for more general reductive groups. For example, in an upcoming paper, the author intends to investigate the case of quaternion groups $B$. In this case, the category $\Mod^\aut(B(\AA)^\times)$ turns out to not quite be symmetric monoidal, but instead be a self-dual module category for $\Mod^\aut(\GL_2(\AA))$. Instead, the author can speculate that there is a coloured operad encoding associative bi-linear operations between the various categories $\Mod^\aut(G)$ for more general reductive groups.
    
    That is, the author expects that instead of the bi-linear operation
    \[
        \oY\co\Mod^\aut(\GL_2(\AA))\otimes\Mod^\aut(\GL_2(\AA))\ra\Mod^\aut(\GL_2(\AA)),
    \]
    one has more sophisticated bi-linear operations relating abstractly automorphic representations for several \emph{different} groups. These operations should possess various associativity constraints, and include various functors induced by the theta correspondence.
\end{remark}

\section{Structure of the Paper} \label{sect:structure_of_paper}

The structure of this paper is as follows. In Part~\ref{part:local_theory}, we will establish the prerequisite local results for the global ones that will follow. Specifically, we will focus on describing the interaction of the parabolic induction functor with the symmetric monoidal structure $\oY$. Our main result will be Theorem~\ref{thm:hat_i_rlax}, which states that the parabolic induction functor (when appropriately normalized) is non-unital right-lax.

In Part~\ref{part:global_theory}, we will take the local theory described above and use it to investigate the category $\Mod^\aut(G)$. We will show that parabolic induction respects the property of being abstractly automorphic in Theorem~\ref{thm:i_aut_to_aut}, and show the decomposition of $\Mod^\aut(G)$ into components in Theorem~\ref{thm:decom_to_comps}.

\part{Local Theory} \label{part:local_theory}

\section{Introduction}

In this part, we will analyze the relationship of the parabolic induction functor with the symmetric monoidal structure $\oY$, focusing solely on the local case.

Specifically, let $F$ be a local field, let $G=\GL_2(F)$ be the general linear group over $F$, let $U=U(F)\subseteq G$ be the subgroup of upper triangular unipotent matrices, and let $T=T(F)\subseteq G$ be the subgroup of diagonal matrices. This data defines the (un-normalized) parabolic induction functor
\[
    i\co\Mod(T)\ra\Mod(G)
\]
between smooth $T$-modules and smooth $G$-modules. The left adjoint to $i$ is denoted by $r\co\Mod(G)\ra\Mod(T)$.

The functors $r,i$ play an important role in the study of the category $\Mod(G)$ (see, e.g., \cite{padic_reps_survey}). We think of $\Mod(T)$ as having the symmetric monoidal structure $\otimes_T$ given by the relative tensor product over $T$, and of $\Mod(G)$ as having the monoidal structure $\oY$ of \cite{abst_aut_reps_arxiv}. It is natural to try to understand the interaction of the adjoint pair $(r,i)$ with these symmetric monoidal structures.

It will turn out that in order to have good monoidal properties, we must first re-normalize the functor $i\co\Mod(T)\ra\Mod(G)$ into a new functor $\widehat{i}\co\Mod(T)\ra\Mod(G)$, which is a twist of $i$. This re-normalization is essentially an incarnation of the L-function that appears in the study of the intertwining operator for $\GL(2)$.

Once we have done so, we will be able to show our main result for this part: the functor $\widehat{i}\co\Mod(T)\ra\Mod(G)$ is non-unital right-lax.

\begin{remark}
    Recall the notion of a non-unital right-lax functor from Recollection~\ref{recall:right_lax}.
    
    In our case, we will use this property in Part~\ref{part:global_theory} because it implies that any object in the image of such a right-lax functor $F$ is automatically enhanced with the structure of a (non-unital) $F(\one_\C)$-module.
\end{remark}

After showing that $\widehat{i}$ is canonically non-unital right-lax, we will then show that $\widehat{i}$ is ``almost'' \emph{unital} right-lax, in the following sense. It turns out that the restriction
\[
    \widehat{i}[\eta^{-1}]\co\Mod(T)[\eta^{-1}]\ra\Mod(G)
\]
to an appropriate localization of $\Mod(T)$ (corresponding to formally inverting the L-function appearing in the normalization factor) is unital right-lax.

As an added bonus, we will construct the intertwining operator as a consequence of the multiplicative structure on $\widehat{i}$. See Section~\ref{sect:intertwining} for a precise formulation.

The structure of this part is as follows. In Section~\ref{sect:notation_and_i}, we formally introduce our notation as well as the functors $r,i$ and their basic properties. In Section~\ref{sect:mult_on_i}, we will define the data that gives the multiplication on $i$, as well as observe the need for the normalization factor. In Section~\ref{sect:hat_i}, we introduce the normalized parabolic induction functor $\widehat{i}$ with its multiplication map. Note that at this point we will not yet know that this multiplication is associative.

In Section~\ref{sect:unit_for_i}, we will discuss the unitality of the multiplication on $\widehat{i}$. We will show that $\widehat{i}$ cannot be unital unless some specific map $\eta$ is inverted in $\Mod(T)$. Additionally, in Section~\ref{sect:assoc}, we will finish the proof that $\widehat{i}$ is non-unital right-lax by showing that its multiplication is associative.

Finally, in Section~\ref{sect:intertwining}, we will show that the multiplication on $\widehat{i}$ induces a self-duality, which turns out to be the intertwining operator. This will let us give a relatively clean description of this intertwining operator and its properties, in terms of a morphism of Frobenius algebras. In particular, we completely avoid the process of analytic continuation.

\section{Prerequisites} \label{sect:notation_and_i}

We will follow the notation of \cite{abst_aut_reps_arxiv}. We begin by recalling this notation in Subsection~\ref{subsect:rep_prereq} below, which will introduce notation for the parabolic induction and restriction functors as well.

Afterwards, we will give some category-theoretical reminders in Subsection~\ref{subsect:cat_prereq}, specifically about the notions of Frobenius algebras and categories. This subsection will only be useful in Section~\ref{sect:intertwining}, where we use these notions to analyze the intertwining operator. Since the rest of the text has little to no dependence on these categorical notions, readers might want to skip Subsection~\ref{subsect:cat_prereq} on a first reading.

\subsection{Representation-Theoretical Prerequisites} \label{subsect:rep_prereq}

For this part, we let $F$ be a non-Archimedean local field of characteristic not equal to $2$. Denote by $\O\subseteq F$ the ring of integers, and let $q$ be the size of its residue field. Let $G=\GL_2(F)$, and select the following subgroups. We set
\[
    U=\left\{\begin{pmatrix}1 & * \\ & 1\end{pmatrix}\right\},
\]
the group of upper triangular unipotent matrices. The subgroup $U\subseteq G$ is normalized by the torus
\[
    T=\left\{\begin{pmatrix}* & \\ & *\end{pmatrix}\right\}\subseteq G.
\]

Denote the category of smooth left $G$-modules by $\Mod(G)=\Mod^L(G)$, and let
\[
    I_{RL}\co\Mod^L(G)\ra\Mod^R(G)
\]
denote the equivalence between smooth left and right $G$-modules given by
\[
    I_{RL}(\pi)(g)\cdot v=\pi(g^T)\cdot v
\]
for $v\in V$ and $(\pi,V)\in\Mod^L(G)$. Denote the inverse of $I_{RL}$ by $I_{LR}$.

For an algebraic variety $X$, we let $S(X)=S(X(F))$ be the space of locally constant functions with compact support. In particular, if $X$ carries a $G$-action, then $S(X)$ is an object of $\Mod(G)$.

We fix a non-trivial additive character $e\co F\ra\CC^\times$, and use it to define
\begin{align*}
    \theta & \co U\ra\CC^\times \\
    \begin{pmatrix}1 & b \\ & 1\end{pmatrix} & \mapsto e(b).
\end{align*}
We denote the functor of $\theta$-co-invariants by $\Phi^-\co\Mod(G)\ra\Vect$.

We give the category $\Mod(G)$ a symmetric monoidal structure as follows. Let $Y=S(\M_2(F)\times F^\times)$ be the $S_3\ltimes G^3$-module defined in \cite{abst_aut_reps_arxiv}, and $\oY$ the corresponding symmetric monoidal structure on $\Mod(G)$:
\[
    A\oY B=\big(I_{RL}(A)\otimes I_{RL}(B)\big)\otimes_{G\times G} Y.
\]
Here, the tensor product is taken over the first and third actions of $G$ on $Y$, with the $G$ action on the result $A\oY B$ coming from the remaining $G$-action on $Y$. Note that the order of the specific $G$-actions used to define $\oY$ is irrelevant, due to the symmetry under $S_3$.

We identify the unit of the symmetric monoidal structure $\oY$ with
\[
    \one_\Ydown=\left\{f\co G\ra\CC\,\middle|\, \substack{\text{$f$ is smooth and} \\ \text{compactly supported mod $U$}},\quad f(ug)=\theta(u)f(g)\right\},
\]
where $G$ acts by right-translation.

\begin{remark}
    For the sake of being self-contained, we explicitly recall some of the actions on $Y$.
    
    First, the actions of the first and third copies of $G$ on $Y$ are given by:
    \[
        (g_1,1,g_3)\cdot\Psi(x,y)=\Psi\left(g_1^{-1}x g_3^{-T},y\cdot\det(g_1g_2^T)\right)
    \]
    for $g_1,g_2\in G$, $y\in F^\times$, $x\in\M_2(F)$ and $\Psi(x,y)\in Y=S(\M_2(F)\times F^\times)$. These two actions are interchanged by the action of the permutation $(1,3)\in S_3$, which is given by:
    \[
        \Psi(x,y)\mapsto\Psi(x^T,y).
    \]
    
    Moreover, the remaining action of $G$ on $Y$ is given by:
    \[
        \left(1,\begin{pmatrix}a & b \\ & 1\end{pmatrix},1\right)\cdot\Psi(x,y)=\abs{a}\cdot e(by\det(x))\cdot\Psi(x,ay)
    \]
    and
    \[
        \left(1,\begin{pmatrix} & -1 \\ 1 & \end{pmatrix},1\right)\cdot\Psi(x,y)=\abs{y}^2\int_{\M_2(F)}\Psi(z,y)\cdot e\left(-y\cdot\left<x,z\right>\right)\d{z}.
    \]
    Here, the bi-linear pairing $\left<x,z\right>$ on $\M_2(F)$ is given by $\tr(x)\tr(z)-\tr(xz)$, and the measure $\d{z}$ on $\M_2(F)$ is normalized such that $e$ is self-dual.
\end{remark}

Recall that the functor $\Phi^-\co\Mod(G)\ra\Vect$ acts as a trace functor, in the sense that the pairing
\[
    \Phi^-(A\oY B)
\]
is naturally isomorphic to the bi-functor $I_{RL}(A)\otimes_G B$, where $\otimes_G$ is the usual relative tensor product over $G$. This bi-functor is a self-duality of $\Mod(G)$. This turns $\Mod(G)$ into a Frobenius algebra in the (higher categorical) sense of $\CC$-linear presentable categories. See also Subsection~\ref{subsect:cat_prereq} below.

Similarly, letting $\Mod(T)$ denote the category of smooth $T$-modules, we give $\Mod(T)$ a symmetric monoidal structure as follows. The product is given by the relative tensor product over $T$:
\[
    A\otimes_T B,
\]
with unit $\one_T$ given by $S(T)=S(F^\times\times F^\times)$.

Let us describe the parabolic induction and restriction functors, which are the main objects of research in this paper.
\begin{definition}
    Let $i\co\Mod(T)\ra\Mod(G)$ be the functor given by:
    \begin{equation*}
        i(V)=\left\{f\co G\ra V\,\middle|\, \text{$f$ is smooth},\quad
        f\left(\begin{pmatrix}a & b \\ & d\end{pmatrix}g\right)=\pi(a,d)\cdot f(g)\right\}
    \end{equation*}
    for $(\pi,V)\in\Mod(T)$. We refer to $i\co\Mod(T)\ra\Mod(G)$ as the \emph{parabolic induction functor}.
\end{definition}

\begin{remark} \label{remark:formulas_for_i}
    We can also define the functor $i\co\Mod(T)\ra\Mod(G)$ via the $G,T$-bi-module $i(\one_T)$, as
    \[
        i(V)=i(\one_T)\otimes_T V.
    \]
    
    Note that the bi-module $i(\one_T)$ is given by the object
    \begin{equation} \label{eq:two_exprs_for_i_one_T}
        S(G)_{/U}(1,-1)\xrightarrow{\sim}S(U\backslash G),
    \end{equation}
    where $V\mapsto V(1,-1)$ denotes the twist of the $T$-action by $(a,d)\mapsto\abs{a/d}$, and where $S(G)_{/U}$ denotes the co-invariants of $S(G)$ under the right action of $U$.
    
    The isomorphism between the two spaces in Equation~\eqref{eq:two_exprs_for_i_one_T} is given by:
    \[
        f(g)\mapsto\int_U f(ug)\d{u}.
    \]
    Indeed, this map respects the $T$-action:
    \[
        \abs{\frac{a}{d}}f\left(\begin{pmatrix}a & \\ & d\end{pmatrix}g\right)\mapsto\abs{\frac{a}{d}}\int_U f\left(\begin{pmatrix}a & \\ & d\end{pmatrix}ug\right)\d{u}=\int_U f\left(u\begin{pmatrix}a & \\ & d\end{pmatrix}g\right)\d{u}.
    \]
\end{remark}

\begin{remark} \label{remark:phi_minus_i_is_forgetful}
    Observe that the composition $\Phi^-\circ i\co\Mod(T)\ra\Vect$ is isomorphic to the forgetful functor:
    \[
        \Phi^-\circ i(V)\cong V,
    \]
    via the isomorphism:
    \[
        f(g)\mapsto\int f\left(\begin{pmatrix} & -1 \\ 1 & \end{pmatrix}\begin{pmatrix}1 & v \\ & 1\end{pmatrix}\right)e(-v)\d{v}.
    \]
    This can be thought of as a trace map on $i$.
\end{remark}

\begin{definition}
    We denote the left adjoint of the functor $i\co\Mod(T)\ra\Mod(G)$ by $r\co\Mod(G)\ra\Mod(T)$, and refer to $r$ as the \emph{parabolic restriction functor}.
\end{definition}

\begin{remark} \label{remark:map_to_i}
    The functor $r\co\Mod(G)\ra\Mod(T)$ is given by the formula
    \[
        r(V)=\prescript{}{U\backslash}{V},
    \]
    which carries a residual $T$-action because $T$ normalizes $U$. Here, the notation $\prescript{}{U\backslash}{V}$ refers to the co-invariants of $V$ under the left action of $U$.
\end{remark}

\begin{remark} \label{remark:tensor_i}
    Because of Remark~\ref{remark:formulas_for_i}, we also have the following fact, for all $A\in\Mod(T)$ and $B\in\Mod(G)$:
    \[
        \Phi^-(i(A)\oY B)=I_{RL}(i(A))\otimes_G B\cong A\otimes_T\left(\prescript{}{U^T\backslash}{B}(1,-1)\right).
    \]
    This is essentially a duality between parabolic induction with respect to $U$ and parabolic induction with respect to $U^T$. That is, the functors of parabolic induction with respect to $U$ and the twist by $(1,-1)$ of the parabolic restriction with respect to $U^T$ are dual under the pairing $\Phi^-(-\oY-)$.
\end{remark}

This can be re-stated more cleanly using the following notation for the normalized action of the Weyl group on $\Mod(T)$:
\begin{definition} \label{def:weyl_twist}
    For a $T$-module $(\pi,V)$, we denote by $V(w)$ the following twist of $V$:
    \begin{align*}
        (\pi,V) & \mapsto(\pi',V(w)) \\
        \pi'(a,d)\cdot v & =\abs{\frac{a}{d}}\pi(d,a)\cdot v.
    \end{align*}
\end{definition}
\begin{remark} \label{remark:i_dual_i_w}
    Remark~\ref{remark:tensor_i} boils down to the claim that the functors $i(V)$ and $i(V(w))$ are dual in the sense of Definition~\ref{def:duality_in_frob_cat} below.
    
    Concretely, if $F\co\Mod(T)\ra\Mod(G)$ and $F'\co\Mod(T)\ra\Vect$ are $\CC$-linear colimit preserving functors, then there is a one-to-one correspondence between natural morphisms:
    \[
        \Phi^-(i(A)\oY F(B))=I_{RL}(i(A))\otimes_G F(B)\ra F'(A\otimes_T B)
    \]
    and natural morphisms
    \[
        F(A)\ra i(F'(A)(w)).
    \]
    Note that $F'(A)$ automatically carries a $T$-module structure.
\end{remark}

\subsection{Categorical Prerequisites} \label{subsect:cat_prereq}

In this subsection, we provide category-theoretical background on the notion of Frobenius algebras, as well as relative Frobenius algebras. Since most of the rest of the text will not be dependent on this subsection, readers may wish to skip it on a first reading.

\begin{definition} \label{def:duality}
    Let $\C$ be a symmetric monoidal category. We say that a pairing
    \[
        \left<\cdot,\cdot\right>\co A\otimes_\C A'\ra\one_\C
    \]
    is a \emph{duality} between two objects $A,A'\in\C$ if the following composition is an isomorphism for all $B,C\in\C$:
    \[
        \Hom(B,C\otimes_\C A')\ra\Hom(B\otimes_\C A,C\otimes_\C A'\otimes_\C A)\xrightarrow{\left<\cdot,\cdot\right>}\Hom(B\otimes_\C A,C).
    \]
\end{definition}

For an object $A$, a pair $(A',\left<\cdot,\cdot\right>)$ as above is called \emph{duality data}. If $A$ has duality data, then it is said to be \emph{dualizable}. The duality data for a dualizable object $A$ is essentially unique, and we denote by $A^\vee$ the \emph{dual object} $A'$ of $A$.

\begin{definition} \label{def:frob_alg}
    Recall that a \emph{Frobenius algebra} in a symmetric monoidal category $\C$ is an algebra object $A\in\C$, with a \emph{trace map} $\lambda\co A\ra\one_\C$, satisfying that the composition of $\lambda$ with the multiplication of $A$ induces a self-duality:
    \[
        A\otimes_\C A\ra\one_\C
    \]
    of $A$.
\end{definition}

Additionally, let $\Pr^L$ be the category of locally presentable categories, with colimit preserving functors between them. The category $\Pr^L$ is a natural setting for the Adjoint Functor Theorem, in the sense that the local presentability condition guarantees that all functors in $\Pr^L$ admit right adjoints.

The category $\Pr^L$ is a symmetric monoidal category, with unit given by the category of sets. The monoidal structure of $\Pr^L$ is the \emph{Lurie tensor product}, defined by the universal property that a functor
\[
    \C\otimes\C'\ra\C''
\]
is a bi-functor which respects colimits separately in each variable.

Note that an algebra object in $\Pr^L$ is just a symmetric monoidal category where the symmetric monoidal structure respects colimits separately in each variable.
\begin{definition}
    A \emph{Frobenius algebra in categories} is a Frobenius algebra in $\Pr^L$.
\end{definition}

There are also relative versions of the above notions. Given any symmetric monoidal category $\C_0$, it is possible to talk about \emph{$\C_0$-linear symmetric monoidal categories}, and \emph{$\C_0$-linear Frobenius algebras in categories} in the obvious manner. The case where $\C_0=\Vect$, the category of $\CC$-vector spaces, will be especially important to us.

If we are given a category $\C$ which is itself a Frobenius algebra in categories, then it makes sense to define another kind of duality on it.
\begin{definition} \label{def:duality_in_frob_cat}
    Let $\C$ be a $\C_0$-linear Frobenius algebra in categories, with trace functor $\Lambda\co\C\ra\C_0$. We say that a pairing
    \[
        \left<\cdot,\cdot\right>\co \Lambda(A\otimes_\C A')\ra\one_{\C_0}
    \]
    is a \emph{relative duality} between two objects $A\in\C$ and $A'\in\C$ if the following composition is an isomorphism for all $B\in\C$ and $C\in\C_0$:
    \begin{multline*}
        \Hom_\C\left(B,C\otimes A'\right)\ra\Hom_\C\left(B\otimes A,C\otimes A'\otimes A\right)\ra \\
        \ra\Hom_{\C_0}\left(\Lambda(B\otimes A),\Lambda(C\otimes A'\otimes A)\right)\xrightarrow{(\spadesuit)} \\
        \xrightarrow{(\spadesuit)}\Hom_{\C_0}\left(\Lambda(B\otimes A),C\otimes\Lambda(A'\otimes A)\right)\xrightarrow{\left<\cdot,\cdot\right>}\Hom_{\C_0}\left(\Lambda(B\otimes A),C\right).
    \end{multline*}
    Here, we used the $\C_0$-linearity of $\Lambda$ for the transition $(\spadesuit)$.
\end{definition}

Note that in the definition above, we have abused notation and used the same symbol $\otimes$ for the symmetric monoidal structure of $\C_0$, the symmetric monoidal structure of $\C$ and the action of $\C_0$ on $\C$.



\begin{definition}
    Let $\C$ be a $\C_0$-linear Frobenius algebra in categories, with trace functor $\Lambda\co\C\ra\C_0$. A \emph{relative Frobenius algebra} in $\C$ is an algebra object $A\in\C$, with a \emph{trace map} $\lambda\co \Lambda(A)\ra\one_{\C_0}$, satisfying that the composition of $\lambda$ with the multiplication of $A$ induces a relative self-duality:
    \[
        \Lambda(A\otimes_\C A)\ra\one_{\C_0}
    \]
    of $A$.
\end{definition}

\begin{remark}
    For the rest of this paper, we will abuse notation and omit the word ``relative'' from the terms ``relative Frobenius algebra'' and ``relative duality''. Moreover, we will always be working with $\Vect$-linear Frobenius algebras and dualities.
\end{remark}

\section{Defining the Multiplication} \label{sect:mult_on_i}

In this section, we will give the construction which induces the data of being non-unital right-lax on (a normalization of) the parabolic induction functor $i\co\Mod(T)\ra\Mod(G)$. Let $\cL\in\Mod(T)$ be defined by:
\begin{definition} \label{def:normalization_factor}
    Let $\cL\in\Mod(T)$ be the space of functions $S(F\times F^\times)$, with $T$-action given by
    \[
        (a,d)\cdot F(\lambda,y)=\abs{\frac{a}{d}}F(d^{-1}\lambda,ady), \qquad\lambda\in F,\, y\in F^\times.
    \]
\end{definition}
\begin{remark} \label{remark:L_is_L_func}
    In the language of \cite{modules_of_zeta_integrals_arxiv}, one can give the space $\cL$ the following informal interpretation. Let us think of the spectrum of $T$ as parametrized by pairs of characters $\chi_1(a)\abs{a}^{s_1}\chi_2(d)\abs{d}^{s_2}$, and let us informally think of $\one_T$ as the ring of regular functions on this spectrum. One can think of elements of $\cL$ as functions on the spectrum of $T$ via the Mellin transform. Thought of like that, the module $\cL$ is generated by the Hecke L-function $L(\chi_1^{-1}\chi_2,s_2-s_1+2)$.
\end{remark}

Our main claim for this section is:
\begin{claim} \label{claim:un_normalized_mult_on_i}
    There is a canonical natural map
    \[
        m\co i(A)\oY i(B)\ra i(\cL\otimes_T A\otimes_T B).
    \]
\end{claim}

Up to the presence of the term $\cL$, this map is the data needed to give a non-unital right-lax structure on $i$. Of course, in order to actually have a right-lax structure, one must show that this data satisfies certain associativity axioms. We will dedicate the rest of this section to proving Claim~\ref{claim:un_normalized_mult_on_i}, as the first step in discussing the right-lax properties of $i$. Along the way, we will also prove a uniqueness result for this multiplicative structure, Proposition~\ref{prop:uniqueness_of_i_mult}, which will show that this choice of multiplication is canonical. Additionally, although the map $i$ will not be \emph{unital} right-lax, we will later use Proposition~\ref{prop:uniqueness_of_i_mult} to show that it does become unital after inverting a certain map.

This multiplication map of Claim~\ref{claim:un_normalized_mult_on_i} is induced via the following construction:
\begin{construction} \label{const:mult_i_one_T}
    We let
    \[
        \mu\co Y\ra\cL
    \]
    be given by
    \[
        \Psi(g,y)\mapsto\Psi\left(\begin{pmatrix}0 & \\ & \lambda\end{pmatrix},\,y\right).
    \]
\end{construction}

The natural map of Claim~\ref{claim:un_normalized_mult_on_i} is induced by the following claim, using Remarks~\ref{remark:map_to_i} and~\ref{remark:tensor_i}:
\begin{claim} \label{claim:i_mult_distribution}
    The map $\mu\co Y\ra\cL$ given in Construction~\ref{const:mult_i_one_T} descends to a map of $T\times T\times T$-modules:
    \[
        \mu\co\prescript{}{U^T\times U\times U^T\backslash}{Y}((1,-1),(0,0),(1,-1))\ra\cL.
    \]
    Here, the notation $-((1,-1),(0,0),(1,-1))$ denotes twisting the three residual $T$-actions, and the three $T$-actions on $\cL$ are the same one.
\end{claim}
\begin{proof}
    It is easy to verify that $\mu$ is $U^T\times U\times U^T$-invariant. Equivariance under $T\times T \times T$ is also immediate.
\end{proof}

\begin{proof}[Proof of Claim~\ref{claim:un_normalized_mult_on_i}]
    Let $A$ and $B$ be two $T$-modules. We tensor the map
    \[
        \mu\co\prescript{}{U^T\times U\times U^T\backslash}{Y}((1,-1),(0,0),(1,-1))\ra\cL
    \]
    by $A$ and $B$, relative to the first and third $T$-actions, respectively. Using Remark~\ref{remark:formulas_for_i}, we obtain a natural map:
    \[
        r\left(i(A)\oY i(B)\right)\ra\cL\otimes_T A\otimes_T B.
    \]
\end{proof}

\subsection{Universal Property of the Multiplication}

When we later discuss the unitality of the right-lax structure on $i$, we will need a slightly stronger variant of Claims~\ref{claim:un_normalized_mult_on_i} and~\ref{claim:i_mult_distribution}. Specifically, it turns out that there is a unique map of the kind constructed in Claim~\ref{claim:i_mult_distribution}. This means that the multiplication map on $i$ is more-or-less unique.
\begin{proposition} \label{prop:uniqueness_of_i_mult}
    The map of $T$-modules
    \[
        \mu\co\prescript{}{U^T\times U\times U^T\backslash}{Y}((1,-1),(0,0),(1,-1))\otimes_{T\times T\times T}\one_T\ra\cL
    \]
    given by forcing the three $T$ actions on $\prescript{}{U^T\times U\times U^T\backslash}{Y}((1,-1),(0,0),(1,-1))$ to be the same is an isomorphism.
\end{proposition}

The meaning of Proposition~\ref{prop:uniqueness_of_i_mult} is that:
\begin{corollary} \label{cor:uniqueness_of_mult}
    Let $F\co\Mod(T)\ra\Mod(T)$ be any functor. Then any natural map:
    \[
        i(A)\oY i(B)\ra i\circ F(A\otimes_T B)
    \]
    factors uniquely through the map
    \[
        m\co i(A)\oY i(B)\ra i(\cL\otimes_T A\otimes_T B)
    \]
    of Claim~\ref{claim:un_normalized_mult_on_i}.
\end{corollary}

\begin{remark}
    It will follow from Proposition~\ref{prop:mult_is_isom} below that, in fact, all functors $F\co\Mod(T)\ra\Mod(G)$ with natural map:
    \[
        i(A)\oY i(B)\ra F(A\otimes_T B)
    \]
    factor uniquely through the map
    \[
        m\co i(A)\oY i(B)\ra i(\cL\otimes_T A\otimes_T B).
    \]
\end{remark}

\begin{proof}[Proof of Proposition~\ref{prop:uniqueness_of_i_mult}]
    Note that the co-invariants of $Y$ under the middle action of $U$ are
    \[
        S(\M_2(F)^{\det=0}\times F^\times).
    \]
    
    Now, we can break this space into orbits under the action of the parabolic subgroup of lower triangular matrices. Specifically, we can give $S(\M_2(F)^{\det=0}\times F^\times)$ a filtration, with graded parts:
    \begin{multline*}
        \left\{\left(\begin{pmatrix}1 & \\ u & 1\end{pmatrix}\begin{pmatrix}\lambda & \\ & 0\end{pmatrix}\begin{pmatrix}1 & v \\ & 1\end{pmatrix},y\right)\right\},\\
        \left\{\left(\begin{pmatrix}1 & \\ u & 1\end{pmatrix}\begin{pmatrix} & \lambda \\ 0 & \end{pmatrix},y\right)\right\}\cup\left\{\left(\begin{pmatrix} & 0 \\ \lambda & \end{pmatrix}\begin{pmatrix}1 & v \\ & 1\end{pmatrix},y\right)\right\},\\
        \left\{\left(\begin{pmatrix}0 & \\ & \lambda\end{pmatrix},y\right)\right\}.
    \end{multline*}
    The first two graded parts have the wrong equivariance properties under the various actions of $T$. Thus, we are left with the last graded part, which gives $\mu$.
\end{proof}

The reader should note that the map of Claim~\ref{claim:un_normalized_mult_on_i} does not quite make $i$ into a non-unital right-lax functor, due to the presence of the term $\cL$. This will be handled in the next section.

\section{Normalized Parabolic Induction} \label{sect:hat_i}

In Section~\ref{sect:mult_on_i}, we defined a natural map
\[
    i(A)\oY i(B)\ra i(\cL\otimes_T A\otimes_T B).
\]
In order to have an actual right-lax functor, we need to modify $i$ to get rid of the extra factor of $\cL$. This will be the focus of this section.

The fact that $i$ needs to be normalized is not too surprising. When working with parabolic induction, one often finds the need to normalize the induction via an L-function. This normalization factor often appears in the functional equation for Eisenstein series in the global theory, and for the intertwining operator in the local theory. It turns out that this normalization factor precisely corresponds to the object $\cL$, in the sense of \cite{modules_of_zeta_integrals_arxiv}. See also Remark~\ref{remark:L_is_L_func}.

Let $\sHom_T\co\Mod(T)^\op\times\Mod(T)\ra\Mod(T)$ be the inner $\Hom$ functor of $\Mod(T)$, which is right adjoint to $\otimes_T$. We set:
\begin{definition}
    Define the \emph{normalized parabolic induction functor}
    \[
        \widehat{i}\co\Mod(T)\ra\Mod(G)
    \]
    by
    \[
        \widehat{i}(V)=i(\sHom_T(\cL,V)).
    \]
    
    We denote its left adjoint, the \emph{normalized parabolic restriction functor}, by $\widehat{r}\co\Mod(G)\ra\Mod(T)$.
\end{definition}

\begin{remark}
    The term ``normalized parabolic induction'' is used in the literature to refer to a different kind of normalization, which we do not use here.
    
    By choosing a square root of the cardinality $q$ of the residue field of $F$, one can simplify the Weyl twist of Definition~\ref{def:weyl_twist}. Letting $\chi=\chi_1\otimes\chi_2\co T\ra\CC^\times$, this means that the intertwining operator can be made to relate the parabolic induction of $\chi_1\otimes\chi_2$ with $\chi_2\otimes\chi_1$, instead of relating $\chi$ with $\chi(w)$.
    
    In this text, we will not make use of this kind of normalization, and we do not choose a square root of $q$.
\end{remark}

\begin{remark}
    By Theorem~A.1 of \cite{modules_of_zeta_integrals_arxiv}, there is a non-canonical isomorphism $\cL\cong\one_T$, meaning that the normalized functor is non-canonically isomorphic to the standard one: $i\cong\widehat{i}$.
\end{remark}

Claim~\ref{claim:un_normalized_mult_on_i} now immediately gives us:
\begin{claim} \label{claim:mult_on_i}
    There is a canonical natural map
    \[
        m\co\widehat{i}(A)\oY \widehat{i}(B)\ra \widehat{i}(A\otimes_T B).
    \]
\end{claim}
In order for this to make $\widehat{i}$ into a non-unital right-lax functor, we still need to verify that this multiplication respects associativity. We will do so in Section~\ref{sect:assoc}.

\section{Unitality} \label{sect:unit_for_i}

Before discussing associativity, let us discuss the unitality of the multiplication on $\widehat{i}\co\Mod(T)\ra\Mod(G)$. Specifically, it turns out that while this functor cannot be unital right-lax, it can come very close to it. We begin by giving a counter-example showing that $\widehat{i}$ cannot be unital, and then show that this counter-example is essentially the only obstruction to the unitality of $\widehat{i}$.

The following is the motivating example for this section.
\begin{example} \label{example:hat_i_not_unital}
    Let $\triv_T(1,-1)$ be the one-dimensional representation of $T$, given by the character $(a,d)\mapsto \abs{a/d}$. This representation satisfies that
    \[
        i(\triv_T(1,-1))
    \]
    has a one-dimensional quotient $\triv_G$, with kernel given by a Steinberg representation $\St$.
    
    Now, the $T$-module $\triv_T(1,-1)$ is clearly a unital commutative algebra with respect to $\otimes_T$. This implies that if
    \[
        \widehat{i}\co\Mod(T)\ra\Mod(G)
    \]
    were unital right-lax, then $i(\triv_T(1,-1))$ would have been a unital commutative algebra. However, let us show that this is impossible.
    
    Indeed, suppose that $i(\triv_T(1,-1))$ was a unital algebra in $\Mod(G)$. Observe that all potential unit maps $\one_\Ydown\ra i(\triv_T(1,-1))$ factor through $\St$. This means that the following isomorphism, required by unitality:
    \[\xymatrix{
        \one_\Ydown\oY i(\triv_T(1,-1)) \ar[r] \ar[rd]_\sim & i(\triv_T(1,-1))\oY i(\triv_T(1,-1)) \ar[d] \\
        & i(\triv_T(1,-1))
    }\]
    factors through
    \[
        \St\oY i(\triv_T(1,-1)),
    \]
    which is $0$. Thus, we obtain a contradiction.
\end{example}

Despite Example~\ref{example:hat_i_not_unital}, the functor $\widehat{i}\co\Mod(T)\ra\Mod(G)$ is \emph{almost} unital right-lax. Essentially, we claim that Example~\ref{example:hat_i_not_unital} is the \emph{only} counter-example. The rest of this section is dedicated to clarifying this statement, formalized in Proposition~\ref{prop:hat_i_almost_unital}.

The key observation is that there is actually a canonical map $\one_\Ydown\ra i(\one_T)$ given as follows.
\begin{construction} \label{const:unit_hat_i_mult}
    Define the map
    \[
        \one_\Ydown\ra S(U\backslash G)\cong i(\one_T)
    \]
    via the formula
    \[
        W(g)\mapsto \int_F W\left(\begin{pmatrix}-1 & \\ & 1\end{pmatrix}\begin{pmatrix} & -1 \\ 1 & \end{pmatrix}^{-1}\begin{pmatrix}1 & b \\ & 1\end{pmatrix}g\right)\d{b}.
    \]
\end{construction}

\begin{remark}
    The map of Construction~\ref{const:unit_hat_i_mult} can also be obtained from the isomorphism $\Phi^-\circ i(V)\cong V$ of Remark~\ref{remark:phi_minus_i_is_forgetful}, by identifying $\Hom(\one_\Ydown,i(\one_T))$ with the roughening of $\Phi^-\circ i(\one_T)=\one_T$.
    
    Similarly, the presence of the factor $\begin{pmatrix}-1 & \\ & 1\end{pmatrix}$ is in order to be consistent with the identification of the functors $I_{RL}(\one_\Ydown)\otimes_G(-)$ and $\Phi^-(-)$ used in Remark~3.17 of \cite{abst_aut_reps_arxiv}. 
\end{remark}

Note that there is a canonical embedding $\eta\co\one_T=S(F^\times\times F^\times)\hookrightarrow\cL$. Our main result for this section is the following ``almost unitality'' result:
\begin{proposition} \label{prop:hat_i_almost_unital}
    The following diagram commutes:
    \[\xymatrix{
        \one_\Ydown\oY\widehat{i}(A) \ar[r] \ar[d]^= & \widehat{i}(\cL)\oY\widehat{i}(A) \ar[d]^m \\
        \widehat{i}(A) \ar[r]^-{\eta} & \widehat{i}(\cL\otimes_T A).
    }\]
    Here, the top horizontal map is induced from the map of Construction~\ref{const:unit_hat_i_mult} via the identification $i(\one_T)=\widehat{i}(\cL)$.
\end{proposition}

Before we discuss the proof of Proposition~\ref{prop:hat_i_almost_unital}, let us describe its consequences.

\begin{remark} \label{remark:hat_i_unital_after_inversion}
    Let $\Mod(T)[\eta^{-1}]$ be the localization of $\Mod(T)$ given by inverting $\eta$. Then $\widehat{i}$ may be restricted to a functor $\widehat{i}[\eta^{-1}]\co\Mod(T)[\eta^{-1}]\ra\Mod(G)$. Now, Proposition~\ref{prop:hat_i_almost_unital} implies that once we show that the multiplication
    \[
        m\co\widehat{i}(A)\oY\widehat{i}(B)\ra\widehat{i}(A\otimes_T B)
    \]
    is associative in Section~\ref{sect:assoc}, then the restriction
    \[
        \widehat{i}[\eta^{-1}]\co\Mod(T)[\eta^{-1}]\ra\Mod(G)
    \]
    will be unital right-lax. This is Corollary~\ref{cor:hat_i_unital_after_inversion}.
\end{remark}

\begin{remark}
    Note that this bypasses the difficulty posed by Example~\ref{example:hat_i_not_unital} because the image of $\triv_T(1,-1)$ under the localization map
    \[
        \Mod(T)\ra\Mod(T)[\eta^{-1}]
    \]
    is $0$. In other words, we claim that Example~\ref{example:hat_i_not_unital} is the only obstruction to the unitality of $\widehat{i}$.
    
    In terms of \cite{modules_of_zeta_integrals_arxiv}, the meaning of inverting $\eta$ is the following. The space $\cL$ is essentially the space of zeta integrals defining the L-function $L(\chi_1^{-1}\chi_2,s_2-s_1+2)$ specified in Remark~\ref{remark:L_is_L_func}. The idea is that the L-function in question has just one pole when thought of as a function on the spectrum of $T$, and the counter-example $\triv_T(1,-1)$ lies in that pole. In particular, inverting $\eta$ kills the counter-example.
    
    Also note that the usual process of meromorphic continuation used to define the intertwining operator (and the L-function $L(\chi_1^{-1}\chi_2,s_2-s_1+2)$) also inverts the map $\eta$. As we will see below, that is essentially the purpose of the analytic continuation.
\end{remark}

\subsection{Proof of Unitality}

We dedicate the rest of this section to the proof of the unitality property Proposition~\ref{prop:hat_i_almost_unital}. The reader may wish to skip the remainder of this section on their first read-through.

We will prove Proposition~\ref{prop:hat_i_almost_unital} by computing the two compositions and comparing the results. The main difficulty which makes this not immediate is to evaluate the isomorphism $\one_\Ydown\oY\widehat{i}(A)\xrightarrow{\sim}\widehat{i}(A)$. This can be done by using the formula given in Remark~3.17 of \cite{abst_aut_reps_arxiv}. However, that formula can only be used after applying the action of a cyclic permutation on the inputs of $Y$ (recall that $Y$ is actually a $S_3\ltimes G^3$-module), which is complicated. Instead of doing this directly, we will use the following alternative strategy. The idea is to directly compute the composition of this permutation of $S_3$ with the map $\mu$ of Construction~\ref{const:mult_i_one_T}. We do this by giving a candidate $\mu'$ for this re-ordering of the inputs of $\mu$, and using the uniqueness principle of Proposition~\ref{prop:uniqueness_of_i_mult} to establish that these two maps are indeed the same.

We begin by constructing the desired re-ordering $\mu'$ of the arguments of $\mu$:
\begin{construction} \label{const:mu_prime}
    We let the map of $T\times T\times T$-modules
    \[
        \mu'\co\prescript{}{U\times U\times U^T\backslash}{Y}((-1,1),(-1,1),(0,0))\ra\cL(-1,2)
    \]
    be given by
    \[
        \Psi(g,y)\mapsto\abs{\lambda}\int\Psi\left(\begin{pmatrix}1 & b \\ & 1\end{pmatrix}\begin{pmatrix}0 & \\ & \lambda\end{pmatrix},\,y\right)\d{b}.
    \]
\end{construction}

\begin{definition} \label{def:partial_fourier}
    Let $\F\co\cL(w)(1,0)\ra\cL$ be the partial Fourier transform, normalized as:
    \[
        \F(F)(\lambda,y)=\abs{y}\int F(\alpha,y)e(-\alpha\lambda y)\d{\alpha}.
    \]
\end{definition}

Denote by $\tau\co S_3\ltimes G^3\ra\End(Y)$ the action on $Y$, so that $\tau((3,2,1))$ is the operation cyclically permuting the three $G$-actions on $Y$, and $\tau(w,w,w)$ denotes the combined application of the matrix $w=\begin{pmatrix} & -1 \\ 1 & \end{pmatrix}$ at all three $G$-actions. Our claim is that $\mu'$ can indeed recover the values of $\mu$ in the following sense:
\begin{claim} \label{claim:mu_prime_same_as_mu}
    The diagram:
    \[\xymatrix{
        Y \ar[d]^\mu \ar[rr]^{\tau((3,2,1))} & & Y \ar[rr]^{\tau(w,w,w)} & & Y \ar[d]^-{\mu'} \\
        \cL & & & & \cL \ar[llll]_\F
    }\]
    commutes.
\end{claim}

Before proving Claim~\ref{claim:mu_prime_same_as_mu}, let us begin by showing how it implies Proposition~\ref{prop:hat_i_almost_unital}.
\begin{proof}[Proof of Proposition~\ref{prop:hat_i_almost_unital}]
    It is sufficient to consider the case $A=\cL$. We will prove the claim by computing and comparing both compositions.
    
    We begin by identifying the domain of both compositions with the space $\overline{Y}$ of co-invariants of $Y$, taken with respect to the character $\theta$ of $U$ with respect to the first action, and taken with respect to the trivial character of $U^T$ with respect to the third action. The space $\overline{Y}$ is isomorphic to the space $\overline{Y}'$, where the $\theta$-co-invariants are taken with respect to the second action, and the $U^T$-co-invariants are taken with respect to the first action. That is, we have an isomorphism:
    \[\xymatrix{
        \overline{Y}' \ar[r]^\sim & \overline{Y}
    }\]
    given by $\tau((3,2,1))^{-1}$. We will compute the two compositions when they are pre-composed with this isomorphism, and composed with the co-unit corresponding to the adjunction $(r,i)$. That is, we work with the following diagram, where the maps between the second and third rows are not maps of $G$-modules, but merely maps of $T$-modules:
    \[\xymatrix{
        \overline{Y}' \ar[rr]^-{\tau((3,2,1))^{-1}} & & \one_\Ydown\oY i(\one_T) \ar[r] \ar[d]^= & i(\one_T)\oY i(\one_T) \ar[d]^m \\
        & & i(\one_T) \ar[r]^-{\eta} \ar[d] & i(\cL) \ar[d] \\
        & & \one_T \ar[r]^-{\eta} & \cL.
    }\]
    
    We begin with the anti-clockwise composition. When we pre-compose it by $\tau((3,2,1))^{-1}$, the isomorphism
    \[
        \one_\Ydown\oY i(\one_T)\cong i(\one_T)
    \]
    is identified with:
    \begin{align*}
        \overline{Y}' & \ra S(U\backslash G) \\
        \Psi(g,y) & \mapsto \int \Psi\left(\begin{pmatrix}1 & 0 \\ u & 1\end{pmatrix}g^{-T},\,\det(g)\right)\d{u}.
    \end{align*}
    Finally, by the adjunction defining $i$, this corresponds to the map:
    \[
        \overline{Y}'\ra\one_T
    \]
    given by:
    \begin{equation} \label{eq:unitality_anti_clockwise}
        \Psi(g,y)\mapsto\int\Psi\left(\begin{pmatrix}1 & 0 \\ u & 1\end{pmatrix}\begin{pmatrix}a^{-1} & \\ & d^{-1}\end{pmatrix},\,ad\right)\d{u}.
    \end{equation}
    This finishes the anti-clockwise composition.
    
    Let us now compute the clockwise composition. Consider $\Psi(g,y)\in\overline{Y}'$. By Claim~\ref{claim:mu_prime_same_as_mu}, applying the clockwise composition to it is the same as applying $\F\circ\mu'\circ\tau(w,w,w)$ to the distribution:
    \[
        \Psi'(g,y)=\Psi(g,\det(g)^{-1})\delta(y\det{g}-1).
    \]
    
    Let $f(g)=\Psi(g,\det(g)^{-1})$. We must apply $\F\circ\mu'\circ\tau(w,w,w)$ to the distribution $\Psi'$. We get:
    \begin{multline*}
        \mu'\circ\tau(w,w,w)(\Psi')(\lambda,y)= \\
        =\abs{\lambda}\int\abs{y}^2\int f(w^{-1}gw^{-1})e\left(-y\left<g,\begin{pmatrix}1 & u \\ & 1\end{pmatrix}\begin{pmatrix}0 & \\ & \lambda\end{pmatrix}\right>\right)\delta(y\det{g}-1)\d{g}\d{u}.
    \end{multline*}
    Simplifying this expression, we have
    \begin{align*}
        \mu'\circ & \tau(w,w,w)(\Psi')(\lambda,y)= \\
        & =\abs{\lambda y^2}\int f\left(\begin{pmatrix}a & b \\ c & d\end{pmatrix}\right)e\left(\lambda y(bu+d)\right)\delta(y(ad-bc)-1)\d{a}\d{b}\d{c}\d{d}\d{u}= \\
        & =\abs{y}\int f\left(\begin{pmatrix}a & 0 \\ c & d\end{pmatrix}\right)e\left(\lambda yd\right)\delta(yad-1)\d{a}\d{c}\d{d}= \\
        & =\int f\left(\begin{pmatrix}1/(yd) & 0 \\ c & d\end{pmatrix}\right)e\left(\lambda yd\right)\d{c}\frac{\d{d}}{\abs{d}}.
    \end{align*}
    Applying $\F$, we get
    \begin{align*}
        \F\circ\mu'\circ\tau(w,w,w)(\Psi') & (\lambda,y)= \\
        & =\abs{y}\int f\left(\begin{pmatrix}1/(yd) & 0 \\ c & d\end{pmatrix}\right)e\left(\alpha y(d-\lambda)\right)\d{c}\frac{\d{d}}{\abs{d}}\d{\alpha} \\
        & =\abs{\lambda}^{-1}\int f\left(\begin{pmatrix}1/(y\lambda) & 0 \\ c & \lambda\end{pmatrix}\right)\d{c} \\
        & =\abs{\lambda^2 y}^{-1}\int f\left(\begin{pmatrix}1 & 0 \\ c & 1\end{pmatrix}\begin{pmatrix}1/(y\lambda) & 0 \\ 0 & \lambda\end{pmatrix}\right)\d{c}.
    \end{align*}
    This corresponds to the map
    \[
        \overline{Y}'\ra\cL
    \]
    given by
    \begin{equation} \label{eq:unitality_clockwise}
        \Psi(g,y)\mapsto\abs{\lambda^2 y}^{-1}\int \Psi\left(\begin{pmatrix}1 & 0 \\ c & 1\end{pmatrix}\begin{pmatrix}1/(y\lambda) & 0 \\ 0 & \lambda\end{pmatrix},\,y\right)\d{c}.
    \end{equation}
    
    Clearly, the map~\eqref{eq:unitality_anti_clockwise} matches up with the map~\eqref{eq:unitality_clockwise} up to:
    \begin{align*}
        \eta\co\one_T & \ra\cL \\
        f(a,d) & \mapsto\abs{\lambda^2 y}^{-1}f(\lambda y,\lambda^{-1}),
    \end{align*}
    as we wanted to show.
\end{proof}

Finally, we end this section with the proof of Claim~\ref{claim:mu_prime_same_as_mu}.
\begin{proof}[Proof of Claim~\ref{claim:mu_prime_same_as_mu}]
    Because of the uniqueness result Proposition~\ref{prop:uniqueness_of_i_mult}, the composition
    \[
        \F\circ\mu'\circ\tau(w,w,w)\circ\tau((3,2,1))
    \]
    factors through $\mu$, and thus defines a map
    \[
        \cL\ra\cL.
    \]
    We must check that this map is the identity map. Therefore, it is sufficient to test this on the indicators of sufficiently small neighborhoods of $\lambda=y=1$.
    
    Thus, we suppose without loss of generality that the character $e\co F\ra\CC^\times$ is unramified. Pick some $\varepsilon\in\O$ of sufficiently small absolute value. Set $\Psi(g,y)$ to be the indicator function of the set:
    \[
        \left\{\left(\begin{pmatrix}a & b \\ c & d\end{pmatrix},y\right)\,\middle|\,\substack{a\in\varepsilon\O,\,b\in\varepsilon\O, \\ c\in\varepsilon\O,\,d\in 1+\varepsilon\O},\,y\in 1+\varepsilon\O\right\}.
    \]
    It is clear that
    \[
        \mu(\Psi)(\lambda,y)=\one_{1+\varepsilon\O}(\lambda)\one_{1+\varepsilon\O}(y).
    \]
    We must verify that the composition
    \[
        \F\circ\mu'\circ\tau(w,w,w)\circ\tau((3,2,1))(\Psi)
    \]
    gives the same function.
    
    Indeed, we obtain that 
    \[
        \frac{1}{\abs{\varepsilon}^2}\tau((3,2,1))(\Psi)
    \]
    is the indicator function of the set
    \[
        \left\{\left(\begin{pmatrix}a & b \\ c & d\end{pmatrix},y\right)\,\middle|\,\substack{a\in\varepsilon^{-1}\O,\,b\in\varepsilon\O, \\ c\in\varepsilon^{-1}\O,\,d\in 1+\varepsilon\O},\,y\in 1+\varepsilon\O\right\}.
    \]
    Moreover, applying $\tau(w,w,w)$, we get:
    \[
        \tau(w,w,w)\circ\tau((3,2,1))(\Psi)\left(\begin{pmatrix}a & b \\ c & d\end{pmatrix},y\right)=\begin{cases}
            \abs{\varepsilon}^2 e(d) & \substack{a\in\varepsilon\O,\,b\in\varepsilon^{-1}\O, \\ c\in\varepsilon\O,\,d\in\varepsilon^{-1}\O},\,y\in 1+\varepsilon\O \\
            0 & \text{otherwise}.
        \end{cases}
    \]
    We apply $\mu'$ to obtain:
    \[
        \mu'\circ\tau(w,w,w)\circ\tau((3,2,1))(\Psi)(\lambda,y)=\abs{\varepsilon}\one_{\varepsilon^{-1}\O}(\lambda)e(\lambda)\one_{1+\varepsilon\O}(y).
    \]
    Finally, using $\F$ gives:
    \[
        \one_{1+\varepsilon\O}(\lambda)\one_{1+\varepsilon\O}(y),
    \]
    as we wanted to show.
\end{proof}

\section{Associativity} \label{sect:assoc}

In this section, we will prove the associativity of the multiplication on
\[
    \widehat{i}\co\Mod(T)\ra\Mod(G)
\]
constructed in Claim~\ref{claim:mult_on_i}.

Specifically, our main theorem for this section is:
\begin{theorem} \label{thm:hat_i_rlax}
    The functor $\widehat{i}\co\Mod(T)\ra\Mod(G)$ is non-unital right-lax symmetric monoidal.
\end{theorem}

As discussed in Remark~\ref{remark:hat_i_unital_after_inversion} above, Theorem~\ref{thm:hat_i_rlax} immediately implies that:
\begin{corollary} \label{cor:hat_i_unital_after_inversion}
    The restriction $\widehat{i}[\eta^{-1}]\co\Mod(T)[\eta^{-1}]\ra\Mod(G)$ of $\widehat{i}$ is unital right-lax symmetric monoidal.
\end{corollary}

We dedicate the rest of this section to proving Theorem~\ref{thm:hat_i_rlax}.

Our strategy for proving Theorem~\ref{thm:hat_i_rlax} is as follows. We are essentially trying to show that the two multiplication maps
\[
    \widehat{i}(\one_T)\oY\widehat{i}(\one_T)\oY\widehat{i}(\one_T)\ra\widehat{i}(\one_T)
\]
are the same. We will do so by showing that applying $\Phi^-$ sends the two maps into the same one, which will follow immediately from the following result, Proposition~\ref{prop:phi_minus_assoc_unique}. Note that the triple-functors $\Phi^-\left(\widehat{i}(A)\oY\widehat{i}(B)\oY\widehat{i}(C)\right)$ and $\Phi^-\left(\widehat{i}(A\otimes_T B\otimes_T C)\right)$ are symmetric with respect to permutations of $A,B,C$.
\begin{proposition} \label{prop:phi_minus_assoc_unique}
    Let $\phi$ be a natural map of functors:
    \[
        \phi\co\Phi^-\left(\widehat{i}(A)\oY\widehat{i}(B)\oY\widehat{i}(C)\right)\ra\Phi^-\left(\widehat{i}(A\otimes_T B\otimes_T C)\right).
    \]
    Then $\phi$ commutes with the above action of $S_3$.
\end{proposition}

Before proving Proposition~\ref{prop:phi_minus_assoc_unique}, let us finish the details of the proof.
\begin{proof}[Proof of Theorem~\ref{thm:hat_i_rlax}]
    We have two compositions
    \[\xymatrix{
        \widehat{i}(A)\oY\widehat{i}(B)\oY\widehat{i}(C) \ar@<2pt>[rr]^{m\circ(m\oY\id)} \ar@<-2pt>[rr]_{m\circ(\id\oY m)} & & \widehat{i}(A\otimes_T B\otimes_T C)
    }\]
    that are given by re-ordering each other's inputs. Therefore, Proposition~\ref{prop:phi_minus_assoc_unique} implies that:
    \[
        \Phi^-(m\circ(m\oY\id))=\Phi^-(m\circ(\id\oY m)).
    \]
    
    Because $\Phi^-$ is exact, the image of $m\circ(m\oY\id)-m\circ(\id\oY m)$ is killed by $\Phi^-$ and is a degenerate representation. Because of Remark~\ref{remark:formulas_for_i}, it is enough to show that the two maps
    \[\xymatrix{
        \widehat{i}(\one_T)\oY\widehat{i}(\one_T)\oY\widehat{i}(\one_T) \ar@<2pt>[rr]^-{m\circ(m\oY\id)} \ar@<-2pt>[rr]_-{m\circ(\id\oY m)} & & \widehat{i}(\one_T)
    }\]
    are equal. However, the image of their difference is a degenerate representation, which must be $0$, as $\widehat{i}(\one_T)$ has no vector invariant to $\SL_2(F)$.
\end{proof}

The proof of Proposition~\ref{prop:phi_minus_assoc_unique} comprises the rest of this section. We will prove the proposition using a uniqueness result. That is, we will show that all natural maps of the requisite form factor through a specific, universal one. The specific map we obtain will be manifestly symmetric with respect to the $S_3$ action on $Y$, which will prove the result.

Our first goal is to explicitly write down the map we are after. To do so, we need an appropriate target space.
\begin{construction}
    Consider the multiplicative convolution product:
    \[
        \cL\otimes_T\cL\ra\widetilde{\cL},
    \]
    where $\widetilde{\cL}$ is the contragradient. Denote the image by $\cL_2$.
\end{construction}
\begin{remark}
    We have an isomorphism:
    \[
        \cL\otimes_T\cL\xrightarrow{\sim}\cL_2.
    \]
    Moreover, recall that $\cL$ consists of functions of $(\lambda,y)$ that are locally constant near $\lambda=0$. Then $\cL_2$ also allows functions that grow logarithmically as $\nu(\lambda)$ near $0$.
    
    Under the Mellin transform, the $T$-module $\cL_2$ allows functions that have a double pole at a specific location.
\end{remark}

We will now construct the desired universal map
\[
    \Phi^-\left(\widehat{i}(A)\oY\widehat{i}(B)\oY\widehat{i}(C)\right)\ra\Phi^-\left(\widehat{i}(A\otimes_T B\otimes_T C)\right).
\]
We will do so by identifying the functor $\Phi^-\left(\widehat{i}(A)\oY\widehat{i}(B)\oY\widehat{i}(C)\right)$ with
\[
    \left(\widehat{i}(A)\otimes\widehat{i}(B)\otimes\widehat{i}(C)\right)\otimes_{G\times G\times G}Y,
\]
which means that all we have to do is construct a map out of $Y$ with the appropriate equivariance properties.

Our candidate is the map:
\begin{construction}
    We let the map of $T\times T\times T$-modules
    \[
        \mu''\co\prescript{}{U\times U\times U\backslash}{Y}((-1,1),(-1,1),(-1,1))\ra\cL_2(-1,2)
    \]
    be given by
    \[
        \Psi(g,y)\mapsto\abs{\lambda}\int\Psi\left(\begin{pmatrix}1 & b \\ & 1\end{pmatrix}\begin{pmatrix}0 & \\ & \lambda\end{pmatrix}\begin{pmatrix}1 & \\ b' & 1\end{pmatrix},\,y\right)\d{b}\d{b'}.
    \]
\end{construction}
We now claim that
\begin{proposition} \label{prop:uniqueness_of_assoc}
    The map of $T$-modules
    \[
        \mu''\co\prescript{}{U\times U\times U\backslash}{Y}((-1,1),(-1,1),(-1,1))\otimes_{T\times T\times T}\one_T\ra\cL_2(-1,2)
    \]
    given by forcing the three $T$ actions on $\prescript{}{U\times U\times U\backslash}{Y}((-1,1),(-1,1),(-1,1))$ to be the same is an isomorphism.
\end{proposition}
\begin{proof}
    The proof is essentially identical to the proof of Proposition~\ref{prop:uniqueness_of_i_mult}.
\end{proof}

This immediately implies that the following construction is universal.
\begin{construction}
    Consider the composition
    \[
        \F_2\circ\mu''\circ\tau(w,w,w)\co\prescript{}{U^T\times U^T\times U^T\backslash}{Y}((1,-1),(1,-1),(1,-1))\ra\cL_2,
    \]
    with $\F_2\co\cL_2(w)(1,0)\xrightarrow{\sim}\cL_2$ the partial Fourier transform on $\cL_2$, given by:
    \[
        \F\otimes\F\co\cL(w)(1,0)\otimes_T\cL(w)(1,0)\xrightarrow{\sim}\cL\otimes_T\cL.
    \]
    Here, $\tau\co S_3\ltimes G^3\ra\End(Y)$ is the action on $Y$ and $\F\co\cL(w)(1,0)\xrightarrow{\sim}\cL$ is as in Definition~\ref{def:partial_fourier}.
    
    Applying the functor $\left(A\otimes B\otimes C\right)\otimes_{T\times T\times T}(-)$ to the composition $\F_2\circ\mu''\circ\tau(w,w,w)$ gives a natural map:
    \[
        \phi_0\co\Phi^-\left(\widehat{i}(A)\oY\widehat{i}(B)\oY\widehat{i}(C)\right)\ra\Phi^-\left(\widehat{i}(A\otimes_T B\otimes_T C)\right).
    \]
\end{construction}

It immediately follows from Proposition~\ref{prop:uniqueness_of_assoc} that:
\begin{corollary} \label{cor:uniqueness_of_assoc}
    Let $F\co\Mod(T)\ra\Vect$ be any functor. Then any natural map:
    \[
        \phi\co\Phi^-\left(\widehat{i}(A)\oY\widehat{i}(B)\oY\widehat{i}(C)\right)\ra F(A\otimes_T B\otimes_T C)
    \]
    factors uniquely through the map
    \[
        \phi_0\co\Phi^-\left(\widehat{i}(A)\oY\widehat{i}(B)\oY\widehat{i}(C)\right)\ra\Phi^-\left(\widehat{i}(A\otimes_T B\otimes_T C)\right)
    \]
    above.
\end{corollary}

Finally, we can prove Proposition~\ref{prop:phi_minus_assoc_unique}:
\begin{proof}[Proof of Proposition~\ref{prop:phi_minus_assoc_unique}]
    We need to show that all natural maps $\phi$ as in the claim are symmetric under the action of $S_3$. We begin by observing that it is sufficient to show that they are all symmetric under the action of a single transposition.
    
    Indeed, the image $\sigma(\phi)$ of $\phi$ under any permutation $\sigma\in S_3$ is still such a natural map. Thus, if all such maps $\sigma(\phi)$ are invariant under a specific transposition, then they are invariant under all of its conjugates as well. But the conjugates of a transposition in $S_3$ generate it.
    
    Therefore, by Corollary~\ref{cor:uniqueness_of_assoc}, it is enough to show that the specific map $\phi_0$ is invariant under a specific transposition of our choice. We will show that $\phi_0$ is symmetric under the action of the transposition exchanging the left and right actions of $G$ on $Y$. Note that this is the transposition of $\M_2(F)$ inside $Y=S(\M_2(F)\times F)$.
    
    However, because $\mu''$ is clearly invariant under the transposition of $\M_2(F)$, the desired result is proven.
\end{proof}

\section{Intertwining Operator} \label{sect:intertwining}

Our goal for this section is to understand the intertwining operator in terms of the multiplicative structure on $\widehat{i}$. Let us begin by an informal explanation of our approach. In this section, we will make extensive use of the notions introduced in Subsection~\ref{subsect:cat_prereq}.

\subsection{Intertwining as Duality}

In this subsection, we will give an informal explanation of how the intertwining operator can arise out of the multiplicative structure of $\widehat{i}$. We will formalize and prove this explanation in Subsection~\ref{subsect:intertwining_construction} below.

The description of our approach is cleanest using the language of the \emph{Day convolution product} on the category of $\CC$-linear colimit preserving functors
\[
    \Fun^L(\Mod(T),\Mod(G)).
\]
Specifically, it turns out that one can assign $\Fun^L(\Mod(T),\Mod(G))$ a symmetric monoidal structure $\oDay$, satisfying the universal property that a map:
\[
    F\oDay F'\ra F''
\]
is the same as a natural morphism
\[
    F(A)\oY F'(B)\ra F''(A\otimes_T B).
\]
This symmetric monoidal structure has the property that its algebras are precisely the right-lax functors $\Mod(T)\ra\Mod(G)$.
 
For the moment, let us ignore the various normalization factors of $\cL$. Our constructions so far turn $\widehat{i}$ into a right-lax symmetric monoidal functor. This means that we have a unital, commutative and associative multiplication map
\[
    m_\mathrm{Day}\co\widehat{i}\oDay\widehat{i}\ra\widehat{i}.
\]
However, we also have a trace map on the algebra $\widehat{i}$, via Remark~\ref{remark:phi_minus_i_is_forgetful} (recall that we are ignoring factors of $\cL$, for the moment):
\[
    \Phi^-\circ\widehat{i}(A)\cong A.
\]
It makes sense to talk about trace maps, because the category
\[
    \Fun^L(\Mod(T),\Mod(G))
\]
carries a trace map of its own. That is, we observe that $\Fun^L(\Mod(T),\Mod(G))$ carries an action of the symmetric monoidal category $\Fun^L(\Mod(T),\Vect)$, equipped with its own Day convolution product. The trace functor
\[
    \Phi^-\co\Mod(G)\ra\Vect
\]
now turns $\Fun^L(\Mod(T),\Mod(G))$ into a Frobenius algebra in the ($2$-)category of $\Fun^L(\Mod(T),\Vect)$-linear presentable categories.

Thus, the functor $\widehat{i}$ becomes a commutative unital algebra with a trace map in the category $\Fun^L(\Mod(T),\Mod(G))$. In particular, it has a bi-linear pairing given by composing the multiplication map with the trace map:
\begin{equation} \label{eq:pairing_on_i}
    \Phi^-\circ(\widehat{i}\oDay\widehat{i})(A)\ra\Phi^-\circ\widehat{i}(A)\cong A.
\end{equation}
Now, Remark~\ref{remark:i_dual_i_w} implies that the functor $\widehat{i}$ is dualizable with respect to the pairing $\Phi^-(-\oDay-)$, and its dual is the functor $A\mapsto\widehat{i}(A(w))$.

Once again ignoring factors of $\cL$, the pairing \eqref{eq:pairing_on_i} induces a map between the functor $A\mapsto\widehat{i}(A)$ and its dual $A\mapsto\widehat{i}(A(w))$. This map turns out to be (a re-parametrization of) the intertwining operator for $\GL(2)$. Moreover, since $A\mapsto A(w)$ is a symmetric monoidal functor, the composition $A\mapsto\widehat{i}(A(w))$ is also an algebra with respect to $\oDay$. Our goal in this section will be to essentially show that this map between $A\mapsto\widehat{i}(A)$ and $A\mapsto\widehat{i}(A(w))$ is an isomorphism of algebras (up to $\cL$).

This will show that the trace pairing on $\widehat{i}$ is non-degenerate, and thus that it is a Frobenius algebra. We will also show that the intertwining operator respects the trace map, which will automatically give the usual result that it is self-inverse up to factors of L-functions.

\subsection{Construction of the Intertwining Operator} \label{subsect:intertwining_construction}

Let us re-state things more formally. The end result is made much more cumbersome by the presence of the normalization factors $\cL$. Therefore, before stating our main theorem for this section, Theorem~\ref{thm:inter}, we give the following corollary:
\begin{corollary} \label{cor:inter_map_of_algs_after_loc}
    The unital commutative algebra
    \[
        \widehat{i}[\eta^{-1},\eta(w)^{-1}]\in\Fun^L\left(\Mod(T)[\eta^{-1},\eta(w)^{-1}],\Mod(G)\right)
    \]
    along with the trace map $\Phi^-\circ\widehat{i}[\eta^{-1},\eta(w)^{-1}](A)\cong A$ is a Frobenius algebra.
    
    Moreover, the functor $\widehat{i}[\eta^{-1},\eta(w)^{-1}]$ is dual to its twist by $A\mapsto A(w)$ as a Frobenius algebra.
\end{corollary}
Here, $\eta(w)$ is the image of $\eta$ under the functor $A\mapsto A(w)$ of Definition~\ref{def:weyl_twist}. Moreover, we assign an algebra structure and trace map to
\[
    \widehat{i}(w)[\eta^{-1},\eta(w)^{-1}]
\]
by transporting the algebra structure and trace functional of $\widehat{i}[\eta^{-1},\eta(w)^{-1}]$ along $A\mapsto A(w)$.

In order to state Theorem~\ref{thm:inter}, let us first more formally define the relevant Frobenius structures, and construct the intertwining operator.
\begin{definition} \label{def:day_conv}
    We denote the \emph{Day convolution product} on the categories of $\CC$-linear colimit preserving functors
    \[
        \Fun^L(\Mod(T),\Mod(G)),\qquad \Fun^L(\Mod(T),\Vect)
    \]
    by $\oDay$.
    
    We think of $\Fun^L(\Mod(T),\Mod(G))$ as a $\Fun^L(\Mod(T),\Vect)$-linear category.
\end{definition}

\begin{remark} \label{remark:day_conv}
    For the sake of being explicit, we can also give the following description. If
    \[
        F\co\Mod(T)\ra\Mod(G),\qquad F'\co\Mod(T)\ra\Mod(G)
    \]
    are given by
    \[
        A\mapsto M\otimes_T A,\qquad A\mapsto M'\otimes_T A
    \]
    respectively, then $F\oDay F'$ is given by
    \[
        A\mapsto (M\oY M')\otimes_{T\times T}\one_T\otimes_T A.
    \]
    
    It is possible to give a similar explicit description for the Day convolution $F\oDay F'$ when both of $F,F'$ lie in $\Fun^L(\Mod(T),\Vect)$, or when one lies in $\Fun^L(\Mod(T),\Vect)$ and the other lies in $\Fun^L(\Mod(T),\Mod(G))$. 
\end{remark}

\begin{remark}
    The map $m$ of Claim~\ref{claim:mult_on_i} gives an associative multiplication map:
    \[
        m_\mathrm{Day}\co\widehat{i}\oDay\widehat{i}\ra\widehat{i}.
    \]
    This turns $\widehat{i}$ into a non-unital commutative algebra in $\Fun^L(\Mod(T),\Mod(G))$.
\end{remark}

\begin{definition}
    We give the category $\Fun^L(\Mod(T),\Mod(G))$ the structure of a Frobenius algebra in $\Fun^L(\Mod(T),\Vect)$-linear presentable categories inherited from the Frobenius algebra structure of $\Mod(G)$.
\end{definition}

\begin{remark}
    More explicitly, we equip $\Fun^L(\Mod(T),\Mod(G))$ with the trace functor
    \[
        \Phi^-\circ-\co\Fun^L(\Mod(T),\Mod(G))\ra\Fun^L(\Mod(T),\Vect).
    \]
    This functor has the property that its composition with the multiplication
    \begin{multline*}
        \oDay\co \\
        \Fun^L(\Mod(T),\Mod(G))\otimes_{\Fun^L(\Mod(T),\Vect)}\Fun^L(\Mod(T),\Mod(G)) \\
        \ra\Fun^L(\Mod(T),\Mod(G))
    \end{multline*}
    induces a self-duality on $\Fun^L(\Mod(T),\Mod(G))$ relative to the base category $\Fun^L(\Mod(T),\Vect)$.
\end{remark}

We now construct the intertwining operator, keeping careful track of factors of $\cL$.
\begin{construction} \label{const:normalized_intertwining}
    The trace map:
    \[
        \Phi^-\circ\widehat{i}(V)\cong\sHom_T(\cL,V),
    \]
    and the multiplication
    \[
        m\co\widehat{i}(A)\oY\widehat{i}(B)\ra\widehat{i}(A\otimes_T B)
    \]
    induce a duality:
    \[
        M\co\widehat{i}(A)\ra\widehat{i}(\cL\otimes_T A(w))
    \]
    using Remark~\ref{remark:i_dual_i_w}.
\end{construction}

Our main theorem for this section is:
\begin{theorem} \label{thm:inter}
    The following diagrams commute.
    \begin{enumerate}
        \item \label{item:inter_self_inv} The intertwining operator is self-inverse:
        \[\xymatrix{
            \widehat{i}(A) \ar[r]^-M \ar@/_2pc/[rr]_{\eta\otimes\eta(w)} & \widehat{i}(\cL\otimes_T A(w)) \ar[r]^-M & \widehat{i}(\cL\otimes_T\cL(w)\otimes_T A).
        }\]
        \item \label{item:inter_trace} The intertwining operator does not affect traces:
        \[\xymatrix{
            \Phi^-\circ\widehat{i}(A) \ar[r] \ar[d]^{\Phi^-(M)} & \sHom_T(\cL,A) \ar[r]^-{\eta} & \sHom_T(\cL,\cL\otimes_T A) \ar[d]^{=} \\
            \Phi^-\circ\widehat{i}(\cL\otimes_T A(w)) \ar[rr] & & A(w).
        }\]
        \item \label{item:inter_unit} The image of the unit $\one_\Ydown$ is invariant to the intertwining operator:
        \[\xymatrix{
            \one_\Ydown \ar[d]^\id \ar[r] & \widehat{i}(\cL) \ar[rr]^-{\widehat{i}(\cL\otimes_T\eta(w))} & & \widehat{i}(\cL\otimes_T\cL(w)) \ar[d]^\id \\
            \one_\Ydown \ar[r] & \widehat{i}(\cL) \ar[rr]^-M & & \widehat{i}(\cL\otimes_T\cL(w)).
        }\]
        \item \label{item:inter_mult} The intertwining operator is an algebra map:
        \[\xymatrix{
            \widehat{i}(A)\oY\widehat{i}(B) \ar[dd]^{M\oY M} \ar[r]^m & \widehat{i}(A\otimes_T B) \ar[d]^M \\
            & \widehat{i}(\cL\otimes_T(A\otimes_T B)(w)) \ar[d]^{\eta} \\
            \widehat{i}(\cL\otimes_T A(w))\oY\widehat{i}(\cL\otimes_T B(w)) \ar[r]^m & \widehat{i}(\cL\otimes_T\cL\otimes_T(A\otimes_T B)(w)).
        }\]
    \end{enumerate}
\end{theorem}

Before proving Theorem~\ref{thm:inter}, there are two things that we need. The first is to have an explicit description of the map
\[
    M\co\widehat{i}(A)\ra\widehat{i}(\cL\otimes_T A(w))
\]
of Construction~\ref{const:normalized_intertwining}. The second will be a variant of Proposition~\ref{prop:uniqueness_of_i_mult}.

We begin with the description of the intertwining operator $M$. The reader should note that the result is slightly different from the usual normalization of the intertwining operator. Specifically, we claim that this map is the Fourier transform of the usual expression:
\begin{claim} \label{claim:explicit_intertwiner}
    The map
    \[
        M\co i(\one_T)\ra i(\cL(w))
    \]
    is given explicitly by the formula:
    \begin{align*}
        \prescript{}{U\backslash}{i(\one_T)} & =\prescript{}{U\backslash}{S(U\backslash G)}\ra\cL(w) \\
        f(g) & \mapsto\abs{y}\int f\left(\begin{pmatrix} & \frac{1}{t} \\ -yt & \end{pmatrix}\begin{pmatrix}1 & u \\ & 1\end{pmatrix}\right)e\left(-y\lambda t\right)\abs{t}\d{t}\d{u},
    \end{align*}
\end{claim}
\begin{proof}[Proof of Claim~\ref{claim:explicit_intertwiner}]
    Let us prove this. The pairing
    \[
        I_{RL}(i(A))\otimes_G i(B)\ra\cL\otimes_T A\otimes_T B
    \]
    induced by the multiplication is explicitly given by the map
    \[
        \prescript{}{U^T\times U^T\backslash}{S(G)}((1,-1),(1,-1))\ra\cL
    \]
    defined as:
    \begin{equation*}
        f(g^{-1})\mapsto\abs{y}^2\int f(g)\delta(y\det{g}-1)e\left(-y\left<g,\begin{pmatrix}0 & \\ & \lambda\end{pmatrix}\right>\right)\d{g}. 
    \end{equation*}
    
    Let us simplify this expression:
    \begin{align*}
        f(g^{-1}) & \mapsto\abs{y}^2\int f\left(\begin{pmatrix}a & b \\ c & d\end{pmatrix}\right)\delta(y(ad-bc)-1)e\left(-y\lambda a\right)\d{a}\d{b}\d{c}\d{d}= \\
        & =\abs{y}\int f\left(\begin{pmatrix}1 & \\ u & 1\end{pmatrix}\begin{pmatrix}a & \\ & \frac{1}{ya}\end{pmatrix}\begin{pmatrix}1 & v \\ & 1\end{pmatrix}\right)e\left(-y\lambda a\right)\abs{a}\d{a}\d{u}\d{v}.
    \end{align*}
    Thus, we get that $M\co i(A)\ra i(\cL(w))$ is induced from the map
    \[
        \prescript{}{U^T\backslash}{i(\one_T)}(1,-1)=\prescript{}{U^T\backslash}{S(U\backslash G)}(1,-1)\ra\cL
    \]
    given by
    \[
        f(g)\mapsto\abs{y}\int f\left(\begin{pmatrix}\frac{1}{a} & \\ & ya\end{pmatrix}\begin{pmatrix}1 & \\ v & 1\end{pmatrix}\right)e\left(-y\lambda a\right)\abs{a}\d{a}\d{v}.
    \]
    This gives the desired result.
\end{proof}

Finally, we need the following variant of the uniqueness result of Proposition~\ref{prop:uniqueness_of_i_mult}: 
\begin{proposition} \label{prop:mult_is_isom}
    The map
    \[
        m_\mathrm{Day}\co\widehat{i}\oDay\widehat{i}\ra\widehat{i}
    \]
    given by the multiplication is an isomorphism.
\end{proposition}
\begin{remark}
    More explicitly, we claim that the map
    \[
        \left(\widehat{i}(\one_T)\oY\widehat{i}(\one_T)\right)\otimes_{T\times T}\one_T\ra\widehat{i}(\one_T)
    \]
    is an isomorphism.
\end{remark}
\begin{proof}[Proof of Proposition~\ref{prop:mult_is_isom}]
    This will follow by proving that the morphism
    \[
        m\co\left(\widehat{i}(\one_T)\oY\widehat{i}(\one_T)\right)\otimes_{T\times T}\one_T\ra\widehat{i}(\one_T)
    \]
    is both injective and surjective.
    
    We begin by observing that $\Phi^-(m)$ is an isomorphism. Indeed, this follows because the pairing
    \[
        \left(I_{RL}(\widehat{i}(\one_T))\otimes_G\widehat{i}(\one_T)\right)\otimes_{T\times T}\one_T\ra\sHom_T(\cL,\one_T)
    \]
    is an isomorphism, which can be seen by the computation of Claim~\ref{claim:explicit_intertwiner}.
    
    From this, it follows that the kernel and co-kernel of
    \[
        m\co\left(\widehat{i}(\one_T)\oY\widehat{i}(\one_T)\right)\otimes_{T\times T}\one_T\ra\widehat{i}(\one_T)
    \]
    are killed by $\Phi^-$.
    
    Let us begin by showing that $m$ is injective. We need to show that the $U$-co-invariants
    \[
        \prescript{}{U^T\times \{1\}\times U^T\backslash}{Y}\otimes_{T\times\{1\}\times T}\one_T
    \]
    have no $\SL_2(F)$-invariant vectors. Now, we note that it is enough to show that
    \[
        \prescript{}{U^T\times U\times U^T\backslash}{Y}\otimes_{T\times\{1\}\times T}\one_T
    \]
    has no vector invariant under (a twist of) $T\cap\SL_2(F)$. This space can be decomposed into orbits via the same argument as in Proposition~\ref{prop:uniqueness_of_i_mult}, which shows that there are no such invariant vectors.
    
    It remains to prove surjectivity. Because $\Phi^-(m)$ is surjective, it is enough to show that $m$ surjects onto the quotient:
    \[
        \widehat{i}(\one_T)_{/\SL_2(F)}=D,
    \]
    where $D=S(F^\times)$ is the $G$-module where $G$ acts via the determinant. Let $N$ be the $T$-module whose underlying space is also $S(F^\times)$, with $T$-action:
    \[
        (a,d)\cdot f(y)=\abs{\frac{a}{d}}f(ady).
    \]
    Then we have a factorization:
    \[\xymatrix{
        \widehat{i}(\one_T)\oY\widehat{i}(\one_T) \ar[d]^m \ar@{=}[r] & \widehat{i}(\one_T)\oY\widehat{i}(\one_T) \ar[d]^\eta & \\
        \widehat{i}(\one_T) \ar[r] & \widehat{i}(N) \ar[r] & D.
    }\]
    We observe that the map $\eta$ is induced from:
    \begin{align*}
        \eta\co\prescript{}{U^T\times U\times U^T\backslash}{Y} & \ra N \\
        \Psi(g,y) & \mapsto\Psi(0,y),
    \end{align*}
    and thus factors through
    \[
        \overline{\eta}\co D\oY D\ra\widehat{i}(N).
    \]
    However, it is easy to verify that this last map $\overline{\eta}$ is in fact an isomorphism. This means that $m$ projects onto $\widehat{i}(N)$, and in particular onto $D$.
\end{proof}

We can now prove our main result about the intertwining operator.
\begin{proof}[Proof of Theorem~\ref{thm:inter}]
    Item~\eqref{item:inter_unit} follows by adjunction from Item~\eqref{item:inter_trace}.
    
    Moreover, Item~\eqref{item:inter_mult} can be tested on $A=B=\one_T[\eta^{-1},\eta(w)^{-1}]$. Because of Proposition~\ref{prop:mult_is_isom}, it follows that the map $M$ respects the multiplication if it respects the unit map. Therefore, Item~\eqref{item:inter_mult} follows from Item~\eqref{item:inter_unit}.
    
    Additionally, since the multiplication and trace maps induce the duality $M$, we see that Item~\eqref{item:inter_self_inv} follows from the combination of Items~\eqref{item:inter_trace} and~\eqref{item:inter_mult}.
    
    Hence, it remains to show that the diagram of Item~\eqref{item:inter_trace} commutes. This can be checked on $A=\cL$. Moreover, since for $A=\cL$ all maps in the diagram of Item~\eqref{item:inter_trace} are maps of $T$-modules, it will suffice to verify that the two compositions are equal on elements of $\Phi^-\circ\widehat{i}(\cL)=\Phi^-\circ i(\one_T)\cong\one_T$ that generate $\Phi^-\circ\widehat{i}(\cL)=\one_T$.
    
    Indeed, suppose without loss of generality that $e$ is unramified. Let
    \[
        f_n(g)\in i(\one_T)=S(U\backslash G)
    \]
    be given by the indicator function of the set $U\cdot K_n$, with $n\geq 1$ and
    \[
        K_n=\left\{\begin{pmatrix}a & b \\ c & d\end{pmatrix}\,\middle|\,\substack{\abs{a-1}\leq q^{-n},\,\abs{b}\leq q^{-n} \\
        \abs{c}\leq q^{-n},\,\abs{d-1}\leq q^{-n}}\right\}.
    \]
    Going along the upper horizontal composition in Item~\eqref{item:inter_trace}, the image of $f_n$ is the function:
    \begin{equation*}
        \int f_n\left(\begin{pmatrix} & -a \\ d & \end{pmatrix}\begin{pmatrix}1 & v \\ & 1\end{pmatrix}\right)e(-v)\d{v}=\begin{cases}e(-1/d) & \abs{ad-1}\leq q^{-n},\,\abs{d}=q^{-n} \\
        0 & \text{otherwise.}\end{cases}
    \end{equation*}
    
    On the other hand, the image of the composition along the left vertical map followed by the lower horizontal map is (using Claim~\ref{claim:explicit_intertwiner}):
    \begin{multline*}
        \int f_n\left(\begin{pmatrix}a & \\ & d\end{pmatrix}\begin{pmatrix} & \frac{1}{t} \\ -t & \end{pmatrix}\begin{pmatrix}1 & u \\ & 1\end{pmatrix}\begin{pmatrix} & -1 \\ 1 & \end{pmatrix}\begin{pmatrix}1 & v \\ & 1\end{pmatrix}\right)e(-v-t)\abs{t}\d{t}\d{u}\d{v}= \\
        =\int f_n\left(\begin{pmatrix}\frac{a}{t} & \frac{av}{t} \\ -dtu & dt-dtuv\end{pmatrix}\right)e(-v-t)\abs{t}\d{t}\d{u}\d{v}= \\
        =\int f_n\left(\begin{pmatrix}* & * \\ u & w\end{pmatrix}\right)e\left(-\frac{w-dt}{u}-t\right)\abs{d\cdot u}^{-1}\d{t}\d{u}\d{w}= \\
        =\int f_n\left(\begin{pmatrix}* & * \\ d & w\end{pmatrix}\right)e\left(-w/d\right)\abs{d}^{-1}\d{w}= \\
        =\begin{cases}e(-1/d) & \abs{ad-1}\leq q^{-n},\,\abs{d}=q^{-n} \\
        0 & \text{otherwise.}\end{cases}
    \end{multline*}
    This matches what we wanted to show.
\end{proof}

\part{Global Theory} \label{part:global_theory}

\section{Introduction}

In this part, we will take the local theory developed in Part~\ref{part:local_theory}, and use it to study the global category $\Mod^\aut(\GL_2(\AA))$ of abstractly automorphic representations.

Let us give more details. From now on, we let $F$ denote a global function field. Set $G=\GL_2(\AA)$, and let $T=T(\AA)\subseteq G$ be its subgroup of diagonal matrices. Then one can define a parabolic induction functor
\[
    i\co\Mod(T)\ra\Mod(G),
\]
with a left adjoint $r\co\Mod(G)\ra\Mod(T)$ called parabolic restriction.

We wish to use the adjoint pair $(r,i)$ to study automorphic representations of $\GL_2(\AA)$. We will present two problematic aspects of a na\"ive implementation of this theory, and then show how the notion of abstractly automorphic representations can help rectify them.

Before we do so, let us note the part of the theory that functions well. The parabolic induction functor $i$ is fairly well behaved. It sends ``most'' irreducible representations to irreducible representations. Moreover, the functor $i$ respects automorphicity: if $(\chi,W)$ is an irreducible automorphic $T$-module (i.e., a character $\chi\co T(F)\backslash T(\AA)\ra\CC^\times$), then $i(W)$ is an automorphic representation of $G$.

However, the parabolic restriction functor $r$ is much more badly behaved. If $(\pi,V)$ is an irreducible automorphic representation which is principal series at all places, then $r(V)$ is an infinite dimensional representation (given by a tensor product of the two-dimensional Jacquet modules of $V$ over all places). In particular, the $T$-module $r(V)$ cannot be said to be automorphic in any reasonable sense.

This picture is fairly unsatisfactory, especially in light of how useful the functors $(r,i)$ are for the study of local representation theory.

Moreover, there is another unsatisfactory aspect of the global theory of automorphic representations. As mentioned in the beginning of Part~\ref{part:introduction}, there is an analogy between the local theory of $p$-adic representations and the global theory of automorphic representations. Both theories admit notions of parabolic induction, cuspidality, constant terms and intertwining operators.

This analogy between the local and global theory is fairly imprecise. Our goal for this part is to rectify these two gaps: the lack of good parabolic restriction and the inability to discuss the cuspidal automorphic spectrum in quite the same terms as the supercuspidal local spectrum. We will do so using the notion of abstract automorphicity.

Indeed, we will see below that while the category $\Mod(G)$ is far too big to be of use, its full subcategory $\Mod^\aut(G)$ is extremely well-behaved with respect to questions of cuspidality and parabolic restriction. Having described the issues we intend to tackle, let us dedicate the rest of this section to informally outlining our approach.

We begin by observing that objects of the category $\Mod(T)$ also have a clear notion of being automorphic: a $T$-module $V$ is abstractly automorphic if and only if it is invariant under the rationals $T(F)\subseteq T(\AA)$. With this notion, it turns out that the parabolic induction functor $i\co\Mod(T)\ra\Mod(G)$ indeed respects automorphicity, and restricts to a functor:
\[
    i^\aut\co\Mod^\aut(T)\ra\Mod^\aut(G).
\]

Moreover, it turns out that $i^\aut$ has a left adjoint functor
\[
    r^\aut\co\Mod^\aut(G)\ra\Mod^\aut(T),
\]
which we refer to as \emph{automorphic parabolic restriction}. The functor $r^\aut$ is extremely well-behaved: it kills irreducible cuspidal automorphic representations, and respects finite length. Moreover, automorphic parabolic restriction acts in a way analogous to the local case. If $i(W)$ is an irreducible representation parabolically induced from $(\chi,W)$ with $\chi=\chi_1\otimes\chi_2$ a product of automorphic characters, then $r\circ i(W)$ is a two-dimensional $T$-module, corresponding to $\chi_1\otimes\chi_2$ and $\chi_2\otimes\chi_1$ (up to twist).

The functor $r^\aut$, along with its interaction with the symmetric monoidal structure $\oY$ on $\Mod^\aut(G)$, allows us to decompose the category $\Mod^\aut(G)$ into \emph{components}, in analogy to the local classification. The cuspidal automorphic representations take the role of the supercuspidal representations, and Eisenstein series take the role of principal series representations. The only new phenomenon occurs at the anomalous part of the spectrum, where some objects are killed by $r^\aut$ despite belonging to an Eisenstein component (see also Remark~\ref{remark:subquot_of_eis_not_quot}). The representations referred to as ``cuspidal'' in \cite{reps_for_gl_2_mod_ell} are a local example where a similar situation occurs.

More precisely, it turns out that we have an equivalence
\begin{equation} \label{eq:decomp_mod_G}
    \Mod^\aut(G)=\prod_{c\in C(G)}\Mod^\aut_c(G),
\end{equation}
with $C(G)$ being the disjoint union of the set $C^\Eis(G)$ of unordered pairs of characters $\AA^\times/F^\times\ra\CC^\times$ up to continuous twists $\abs{\cdot}^{s_1}\otimes\abs{\cdot}^{s_2}$, with the set $C^\cusp(G)$ of irreducible cuspidal automorphic representations up to continuous twist $\abs{\det(\cdot)}^s$.

The structure of this part is as follows. In Section~\ref{sect:globalization}, we will make the necessary adjustments to turn the local results of Part~\ref{part:local_theory} about the interaction of $(r,i)$ with the symmetric monoidal structures into global results. In Section~\ref{sect:i_aut_aut}, we will show that the parabolic induction functor sends automorphic representations to automorphic representations. In Section~\ref{sect:decomp}, we will establish the decomposition~\eqref{eq:decomp_mod_G}, and in Section~\ref{sect:results_of_decomp} we will give some corollaries, fulfilling our promises from Section~\ref{sect:GL2}.

\section{Globalization} \label{sect:globalization}

In this section, we will discuss what needs to be done in order to lift the data of $\widehat{i}$ being a non-unital right-lax symmetric monoidal functor at every place to global data.

Therefore, for the rest of this paper, we change our notation. We denote by $F$ a \emph{global} function field, and for every place $v$ we let the subscript $v$ denote the local constructions of the previous part over the local field $F_v$. This means that $T_v$, $G_v$, $\one_{\Ydown,v}$, $\cL_v$, $\widehat{i}_v$ etc. all denote the corresponding local constructions.

For the rest of this paper, we let $T=T(\AA)$, $G=G(\AA)$, and let $\Mod(T)$ and $\Mod(G)$ be the corresponding categories of smooth left modules. We let $\cL=S(\AA\times\AA^\times)$, with $T=T(\AA)$-action given by
\[
    (a,d)\cdot F(\lambda,y)=\abs{\frac{a}{d}}F(d^{-1}\lambda,ady)
\]
as in Definition~\ref{def:normalization_factor}. We similarly define the functor
\[
    \widehat{i}\co\Mod(T)\ra\Mod(G)
\]
as
\[
    \widehat{i}(A)=i(\sHom_T(\cL,A)),
\]
where $i\co\Mod(T)\ra\Mod(G)$ is the usual parabolic induction functor.

Our main claim for this section is
\begin{theorem} \label{thm:i_hat_rlax_global}
    The functor $\widehat{i}\co\Mod(G)\ra\Mod(T)$ is non-unital right-lax symmetric monoidal.
\end{theorem}

This essentially boils down to establishing that the multiplication on $\widehat{i}$ respects the distinguished vectors, which can be checked directly.

It remains to discuss the issue of unitality. Here, it turns out that:
\begin{warning}
    The local map $\one_{\Ydown,v}\ra i_v(\one_{T,v})$ does \emph{not} respect the distinguished vectors.
\end{warning}
In particular, there is no global map $\one_\Ydown\ra i(\one_T)$. However, this makes some amount of sense, as this is not the map we are after anyway: we want $\widehat{i}$ to be the right-lax functor, not $i$. So, consider the localization $\Mod(T)[\eta^{-1}]$ given by inverting the maps
\[
    \eta_v\co\one_{T,v}\hookrightarrow\cL_v
\]
at every place. Then we can uniquely complete the diagram
\[\xymatrix{
    \one_{\Ydown,v} \ar[r] \ar@{-->}[rd] & \widehat{i}_v(\cL_v[\eta_v^{-1}]) \\
    & \widehat{i}_v(\one_{T,v}[\eta_v^{-1}]). \ar[u]^{\eta_v}
}\]
It is easy to verify that the resulting map $\one_{\Ydown,v}\ra\widehat{i}_v(\one_{T,v}[\eta_v^{-1}])$ respects the distinguished vectors, and therefore we obtain:
\begin{theorem} \label{thm:i_hat_unital_rlax_global}
    The restriction $\widehat{i}[\eta^{-1}]\co\Mod(T)[\eta^{-1}]\ra\Mod(G)$ of $\widehat{i}$ is unital right-lax symmetric monoidal.
\end{theorem}

\begin{warning}
    Despite the suggestive notation, there is no actual map
    \[
        \eta\co\one_T\ra\cL
    \]
    that we are inverting for the localization
    \[
        \Mod(T)\ra\Mod(T)[\eta^{-1}].
    \]
    Indeed, the maps $\eta_v\co\one_{T,v}\ra\cL_v$ do not define a global map, as they do not respect the distinguished vectors. In the classical theory, this is worked around via analytic continuation, see also Remark~\ref{remark:analytic_continuation}. However, we will not need to do this.
\end{warning}

\begin{remark}
    There is a similar issue with regard to the trace map on $\widehat{i}$ and the intertwining operator. Indeed, the local trace map
    \[
        \Phi_v^-\circ\widehat{i}_v(V)\cong\sHom_{T_v}(\cL_v,V)
    \]
    of Section~\ref{sect:intertwining} does not turn into a global map, as it does not respect the distinguished vectors. However, after inverting $\eta_v(w)\co\one_{T,v}(w)\ra\cL_v(w)$ at every place $v$, we can uniquely complete the diagram:
    \[\xymatrix{
        \Phi_v^-\circ\widehat{i}_v(V) \ar[r] \ar@{-->}[rd] & \sHom_{T_v}(\cL_v,V[\eta_v(w)^{-1}]) \\
        & \sHom_{T_v}(\cL_v\otimes_{T_v}\cL_v(w),V[\eta_v(w)^{-1}]). \ar[u]^{\eta_v}
    }\]
    It is easy to verify that the resulting map respects the distinguished vectors, giving a natural map:
    \[
        \Phi^-\circ\widehat{i}(A)\ra\sHom_{T}(\cL\otimes_T\cL(w),A[\eta(w)^{-1}])
    \]
    and an intertwining operator:
    \[
        M\co\widehat{i}(A)\ra\widehat{i}(A(w)[\eta^{-1}]).
    \]
\end{remark}

\section{Parabolic Induction and Automorphicity} \label{sect:i_aut_aut}

Our goal in this section is to show that the parabolic induction functor
\[
    \widehat{i}\co\Mod(T)\ra\Mod(G)
\]
is well-behaved with respect to the property of being automorphic. We formalize this using the following notation:
\begin{definition}
    Let $\Mod^\aut(T)$ be the category of smooth $T(\AA)/T(F)$-modules.
\end{definition}
Note that there is a forgetful functor
\[
    \Mod^\aut(T)\ra\Mod(T),
\]
induced by restriction along the map $T(\AA)\ra T(\AA)/T(F)$. This functor is fully faithful, and its essential image consists of the objects of $\Mod(T)$ which are invariant under the action of $T(F)\subseteq T(\AA)$. We will usually identify $\Mod^\aut(T)$ with its essential image in $\Mod(T)$. We refer to objects of $\Mod^\aut(T)$ as \emph{abstractly automorphic} $T(\AA)$-modules.

We give $\cI_T=S(T(F)\backslash T(\AA))$ the structure of a unital commutative algebra using the convolution product over $T$.
\begin{remark}
    The category $\Mod^\aut(T)$ is equivalent to the category $\Mod(\cI_T)$ of $\cI_T$-modules in $\Mod(T)$.
\end{remark}

Let $\cI_G\subseteq S(\GL_2(F)\backslash\GL_2(\AA))$ be the space of compactly supported smooth automorphic functions which are orthogonal to one-dimensional characters. Recall from \cite{abst_aut_reps_arxiv} that $\cI_G$ acquires a unique commutative algebra structure from the canonical unit map $\one_\Ydown\ra\cI_G$ induced by taking Whittaker functions. Also recall that the category $\Mod(\cI_G)$ of $\cI_G$-modules in $\Mod(G)$ is the category $\Mod^\aut(G)$ of \emph{abstractly automorphic} $G$-modules, and that the forgetful functor $\Mod^\aut(G)\ra\Mod(G)$ is fully-faithful.

Our main theorem for this section is thus:
\begin{theorem} \label{thm:i_aut_to_aut}
    The functor $\widehat{i}\co\Mod(T)\ra\Mod(G)$ sends abstractly automorphic objects to abstractly automorphic objects.
\end{theorem}
In other words, we claim that $\widehat{i}$ descends to a functor
\[
    \widehat{i}^\aut\co\Mod^\aut(T)\ra\Mod^\aut(G).
\]

\begin{remark}
    Because $\widehat{i}$ and $i$ are non-canonically isomorphic, Theorem~\ref{thm:i_aut_to_aut} shows that the un-normalized functor $i$ also respects automorphicity, and restricts to:
    \[
        i^\aut\co\Mod^\aut(T)\ra\Mod^\aut(G).
    \]
\end{remark}

\begin{remark} \label{remark:aut_par_rest}
    The fully-faithful forgetful functors $\Mod^\aut(T)\ra\Mod(T)$ and $\Mod^\aut(G)\ra\Mod(G)$ admit left adjoints, given by $\cI_T\otimes_T(-)$ and $\cI_G\oY(-)$, respectively. In particular, the functor $\widehat{i}^\aut$ immediately admits a left-adjoint, given by the co-invariants under the action of $T(F)^\times$:
    \begin{align*}
        \widehat{r}^\aut\co\Mod^\aut(G) & \ra\Mod^\aut(T), \\
        V & \mapsto\widehat{r}(V)_{/T(F)^\times}.
    \end{align*}
\end{remark}

To prove Theorem~\ref{thm:i_aut_to_aut}, it is sufficient to prove the following proposition:
\begin{proposition} \label{prop:hat_i_I_T_is_aut}
    The object $\widehat{i}(\cI_T[\eta^{-1}])$ lies in $\Mod^\aut(G)$.
\end{proposition}

Indeed, this implies the theorem:
\begin{proof}[Proof of Theorem~\ref{thm:i_aut_to_aut}]
    Denote by $\C\subseteq\Mod^\aut(T)$ the full subcategory whose objects are sent by $\widehat{i}$ to $\Mod^\aut(G)$. We wish to prove that $\C=\Mod^\aut(T)$. Note that $\widehat{i}$ is exact and respects colimits. Therefore, since $\Mod^\aut(G)$ is closed under colimits, it follows that $\C$ is closed under colimits as well.
    
    Moreover, since $\Mod^\aut(G)$ is closed under taking sub-objects, it follows that $\C$ has this property as well.
    
    It therefore remains to show that $\cI_T\in\C$. However, Proposition~\ref{prop:hat_i_I_T_is_aut} implies that $\cI_T[\eta^{-1}]$ is in $\C$. Since $\cI_T$ is a sub-object of $\cI_T[\eta^{-1}]$, we are done.
\end{proof}

It remains to show Proposition~\ref{prop:hat_i_I_T_is_aut}. Because Theorem~\ref{thm:i_hat_unital_rlax_global} implies that $\widehat{i}(\cI_T[\eta^{-1}])$ is a unital algebra object in $\Mod(G)$, the following claim will finish the proof:
\begin{claim} \label{claim:hat_i_I_T_through_I_G}
    The unit map
    \[
        \one_\Ydown\ra\widehat{i}(\cI_T[\eta^{-1}])
    \]
    factors through the unit map of $\cI_G$:
    \[
        \one_\Ydown\ra\cI_G.
    \]
\end{claim}
\begin{remark}
    In fact, our proof of Claim~\ref{claim:hat_i_I_T_through_I_G} will show that the unique map completing the diagram
    \[\xymatrix{
        \one_\Ydown \ar[r] \ar[rd] & \cI_G \ar@{-->}[d] \\
        & \widehat{i}(\cI_T[\eta^{-1}])
    }\]
    is the usual constant term map we are familiar with from the theory of Eisenstein series.
\end{remark}

\begin{remark} \label{remark:eisenstein_functional_eqs}
    By Corollary~\ref{cor:inter_map_of_algs_after_loc}, the image of the unit map $\one_\Ydown\ra\widehat{i}(\cI_T[\eta^{-1}])$ is invariant under the intertwining operator:
    \[
        M\co\widehat{i}(\cI_T[\eta^{-1}])\ra\widehat{i}(\cI_T(w)[\eta^{-1},\eta(w)^{-1}])\cong\widehat{i}(\cI_T[\eta^{-1},\eta(w)^{-1}]).
    \]
    Thus, because we interpret the map $\cI_G\ra\widehat{i}(\cI_T[\eta^{-1}])$ of Claim~\ref{claim:hat_i_I_T_through_I_G} as the constant term map, we automatically get the functional equation for Eisenstein series.
\end{remark}

\begin{remark} \label{remark:analytic_continuation}
    It is interesting to note that the above gives a description of the functional equation for Eisenstein series which makes no mention of analytic continuation. It is worthwhile to have an explanation for how this can happen.
    
    The idea is that the role of analytic continuation is to allow a canonical trivialization of the normalization factor $\cL$, which represents an L-function, as in Remark~\ref{remark:L_is_L_func}.
    
    That is, by applying the Mellin transform, one can think of elements of $\cI_T=S(T(\AA)/T(F))$ as analytic functions of two complex variables $(s_1,s_2)\in\CC^2$, parametrized by $\chi_1\abs{\cdot}^{s_1}\otimes\chi_2\abs{\cdot}^{s_2}$. If we consider the extension $\cI_T\subseteq\cI'_T$ corresponding to functions that are analytic only on some right half plane $\Re{s_1}\gg 0$, $\Re{s_2}\gg 0$, then there is a canonical isomorphism
    \[
        \cI'_T\otimes_{\cI_T}\cL\xrightarrow{\sim}\cI'_T.
    \]
    Thus, analytic continuation trivializes $\cL$. See also Section~5 of \cite{modules_of_zeta_integrals_arxiv}.
    
    In our case, by keeping track of the normalization factors, we have successfully avoided the need for analytic continuation.
\end{remark}

\begin{proof}[Proof of Claim~\ref{claim:hat_i_I_T_through_I_G}]
    This is a straightforward verification. Indeed, the map
    \[
        \one_\Ydown\ra\widehat{i}(\cI_T[\eta^{-1}])
    \]
    is given via the map
    \[
        \one_\Ydown\otimes\cL\ra i(\cI_T[\eta^{-1}])
    \]
    which is adjoint to the composition:
    \begin{equation*}
        \prescript{}{U\backslash}{\left(\one_\Ydown\otimes\cL\right)}\ra\one_T[\eta^{-1}]\ra\cI_T[\eta^{-1}].
    \end{equation*}
    
    Now, by unwinding the definitions, we see that the map $\prescript{}{U\backslash}{\left(\one_\Ydown\otimes\cL\right)}\ra\one_T[\eta^{-1}]$ above sends $W(g)\otimes F(\lambda,y)\in\prescript{}{U\backslash}{\left(\one_\Ydown\otimes\cL\right)}$ (with $W(g)$ being left $\theta$-equivariant) to the convolution:
    \begin{equation*}
         \int_{\AA^\times}\int_{\AA^\times}\int_\AA W\left(\begin{pmatrix} & -1 \\ 1 & \end{pmatrix}\begin{pmatrix}1 & u \\ & 1\end{pmatrix}\begin{pmatrix}a^{-1} & \\ & d^{-1}\end{pmatrix}\right)\abs{\frac{a}{d}}F(d^{-1}\lambda,ady)\d{u}\dtimes{a}\dtimes{d}.
    \end{equation*}
    Therefore, it remains to show that the map sending $W(g)\in\one_\Ydown$ to
    \begin{equation} \label{eq:unit_for_i_global}
        \sum_{\alpha,\delta\in F^\times}\int_\AA W\left(\begin{pmatrix} & -1 \\ 1 & \end{pmatrix}\begin{pmatrix}1 & u \\ & 1\end{pmatrix}\begin{pmatrix}\alpha a & \\ & \delta d\end{pmatrix}\right)\d{u}
    \end{equation}
    factors through $\one_\Ydown\ra\cI_G$.
    
    Indeed, we observe that the expression appearing in Equation~\eqref{eq:unit_for_i_global} is the same as:
    \[
        \sum_{\gamma\in U_2(F)\backslash\GL_2(F)}\int_{\AA/F} W\left(\gamma\begin{pmatrix}1 & u \\ & 1\end{pmatrix}\begin{pmatrix}a & \\ & d\end{pmatrix}\right)\d{u}.
    \]
    Since the map $\one_\Ydown\ra\cI_G$ is given by $W(g)\mapsto\sum W(\gamma g)$, we are done.
\end{proof}

\section{Decomposition Into Components} \label{sect:decomp}

Our goal in this section is to decompose the category $\Mod^\aut(G)$ of abstractly automorphic representations into components, in a manner analogous to the local classification.

We will do so as follows. We will begin by separating the category $\Mod^\aut(G)$ into a cuspidal part and an Eisenstein part. We will then break down the cuspidal part into components indexed by irreducible cuspidal automorphic representations up to continuous twist. Finally, we will decompose the Eisenstein part of the category as well.

We begin by separating out the cuspidal part of the category. Let $\cI_G^\cusp\subseteq\cI_G$ be the kernel of the constant term map
\[
    \cI_G\ra\widehat{i}(\cI_T[\eta^{-1}])
\]
constructed in Claim~\ref{claim:hat_i_I_T_through_I_G}. Denote the subspace of $\cI_G$ which is orthogonal to $\cI_G^\cusp$ with respect to the Petersson inner product by $\cI_G^\Eis$. We claim that:
\begin{proposition} \label{prop:decomp_cusp_eis}
    The subspaces $\cI_G^\cusp,\cI_G^\Eis\subseteq\cI_G$ are ideals, and we have a decomposition
    \[
        \cI_G=\cI_G^\cusp\times\cI_G^\Eis
    \]
    of algebras.
\end{proposition}
\begin{proof}
    Because the constant term map is an algebra map, then $\cI_G^\cusp$ is an ideal. Moreover, because the Petersson inner product respects the multiplication on $\cI_G$ (see Remark~4.21 of \cite{abst_aut_reps_arxiv}), we see that $\cI_G^\Eis$ is an ideal as well. It remains to show that
    \[
        \cI_G^\cusp\oplus\cI_G^\Eis
    \]
    is all of $\cI_G$.
    
    Indeed, for every $r>1$ and every compact open set $K\subseteq\GL_2(\AA)$, let
    \[
        S_{<r}^\cusp(\GL_2(F)\backslash\GL_2(\AA)/K)
    \]
    be the space of $K$-smooth cusp forms supported on $g\in G$ with $1/r<\abs{\det(g)}<r$. The desired result follows because $S_{<r}^\cusp(\GL_2(F)\backslash\GL_2(\AA)/K)$ is finite dimensional.
\end{proof}

This shows that $\Mod^\aut(G)$ decomposes into a cuspidal part $\Mod^\cusp(G)=\Mod(\cI_G^\cusp)$ and an Eisenstein part $\Mod^\Eis(G)=\Mod(\cI_G^\Eis)$:
\[
    \Mod^\aut(G)=\Mod^\cusp(G)\times\Mod^\Eis(G).
\]

\begin{remark} \label{remark:eis_is_image}
    As an additional consequence of Proposition~\ref{prop:decomp_cusp_eis}, we see that $\cI_G^\Eis$ is the image of the map $\cI_G\ra\widehat{i}(\cI_T[\eta^{-1}])$.
\end{remark}

We also get the following corollary for free:
\begin{corollary} \label{cor:aut_par_rest_kills_cusp}
    Let $M\in\Mod^\cusp(G)$. Then $\widehat{r}^\aut(M)=0$.
\end{corollary}
\begin{proof}
    We claim that each $\widehat{i}(N)$ is killed by the action of $\cI_G^\cusp$. Indeed, it suffices to prove this for $N=\cI_T[\eta^{-1}]$. The fact that $\cI_G\ra\widehat{i}(\cI_T[\eta^{-1}])$ is an algebra map now finishes the proof.
\end{proof}
\begin{warning}
    Note that unlike the local case, the property $\widehat{r}^\aut(M)=0$ does not characterize the cuspidal representations. Indeed, because automorphic parabolic induction can have subquotients that are neither sub-objects nor quotients (this is part of the \emph{anomalous}, or \emph{non-isobaric}, spectrum), it follows that there are objects $M\in\Mod^\Eis(G)$ with $\widehat{r}(M)=0$. See also Warning~\ref{warn:eisenstein_killed_by_rest}.
\end{warning}

\subsection{Cuspidal Components}

We are left with the task of decomposing each of $\cI_G^\cusp$ and $\cI_G^\Eis$ separately. We begin by decomposing $\cI_G^\cusp$ as follows.

Let $G^1\subseteq G$ be the subgroup of all $g\in G$ with $\abs{\det(g)}=1$. Let $(\pi,V)$ be any irreducible cuspidal automorphic representation of $G^1$, with distinguished generic vector
\[
    v^\gen\in\Phi^-V.
\]
Consider the $G$-module given by induction $\Ind_{G^1}^G V$. Spectrally, the $G$-module $\Ind_{G^1}^G V$ contains all twists $\pi\times\abs{\det(\cdot)}^s$.

\begin{remark} \label{remark:explicit_G_1_induction}
    Choose a place $\nu$ of $F$ of degree $1$, and a uniformizer $\pi_\nu$. This choice induces an isomorphism $\Ind_{G^1}^G V=S(q^\ZZ)\otimes V$, where $q$ is the size of the base field of $F$, and $S(q^\ZZ)$ is the space of compactly supported functions on the discrete set $q^\ZZ$. The group $G$ acts on $S(q^\ZZ)\otimes V$ via:
    \begin{multline*}
        g\cdot(f(y)\otimes v)= \\
        =f(y\abs{\det{g}})\otimes\pi\left(i_v\!\begin{pmatrix}\pi_\nu^{-\log_q(y)} & \\ & 1\end{pmatrix}\cdot g\cdot i_v\!\begin{pmatrix}\pi_\nu^{\log_q(y)+\log_q\abs{\det{g}}} & \\ & 1\end{pmatrix}\right)\cdot v,
    \end{multline*}
    for $g\in G$, $f\in S(q^\ZZ)$, $y\in q^\ZZ$ and $v\in V$, where $i_\nu\co \GL_2(F_\nu)\ra\GL_2(\AA)$ is the embedding corresponding to the place $\nu$..
\end{remark}

We observe that the choice of $v^\gen$ is adjoint to a surjective map $\one_\Ydown\twoheadrightarrow\Ind_{G^1}^G V$, which turns $\Ind_{G^1}^G V$ into a unital commutative algebra. The fact that $V$ is automorphic means that this is in fact an algebra over $\cI_G$:
\[
    \cI_G\ra \Ind_{G^1}^G V.
\]

Let $C^\cusp(G)$ be the set of isomorphism classes of irreducible cuspidal automorphic representations $(\pi,V)$ of $G^1$ up to conjugation by $G$. Observe that $C^\cusp(G)$ is the same as the set of isomorphism classes of irreducible cuspidal automorphic representations of $G$ up to continuous twists $V\mapsto V\otimes\abs{\det(\cdot)}^s$. We now claim that:
\begin{proposition} \label{prop:decomp_cusp}
    The above construction induces an isomorphism of algebras:
    \[
        \cI_G^\cusp\xrightarrow{\sim}\prod_{(\pi,V)\in C^\cusp(G)}\Ind_{G^1}^G V.
    \]
\end{proposition}
\begin{remark} \label{remark:prod_is_coprod}
    Note that the infinite product in Proposition~\ref{prop:decomp_cusp} is taken in $\Mod(G)$, and is therefore also equal to an infinite co-product. That is, this product is actually a direct sum.
\end{remark}
\begin{proof}[Proof of Proposition~\ref{prop:decomp_cusp}]
    For every cuspidal $(\pi,V)\in C^\cusp(G)$, we have a section
    \[
        \Ind_{G^1}^G V\ra\cI_G^\cusp
    \]
    given in the notation of Remark~\ref{remark:explicit_G_1_induction} by
    \[
        \phi(g)\otimes f(y)\mapsto\phi\left(g\cdot i_\nu\!\begin{pmatrix}\pi_\nu^{\log_q\abs{\det{g}}} & \\ & 1\end{pmatrix}\right)f(\abs{\det{g}}),
    \]
    where $\phi(g)$ is an automorphic form in the space of $V$.
    
    Together with the finite dimensionality of the space
    \[
        S_{<r}^\cusp(\GL_2(F)\backslash\GL_2(\AA)/K)
    \]
    from the proof of Proposition~\ref{prop:decomp_cusp_eis} and the orthogonality of the different cuspidal representations, the claim follows.
\end{proof}

\subsection{Eisenstein Components}

It remains to decompose the Eisenstein part of the category $\Mod^\aut(G)$. Our strategy for doing so is as follows. It is immediate that $\widehat{i}(\cI_T[\eta^{-1}])$ decomposes into a direct product over components, because $\cI_T$ does so as well. Because we have a map of algebras $\cI_G\ra\widehat{i}(\cI_T[\eta^{-1}])$, we will be able to define sub-algebras of $\cI_G$ corresponding to each component separately. Thus, it will remain to show that every element $\phi\in\cI_G$ is the direct sum of its projections corresponding to each component. This is a property of $\cI_G$ as a $G$-module, which can be verified because these projections can be done via elements of the center of $\Mod(G)$.

Let us formalize this. Let $C^\Eis(G)$ be the set of automorphic characters $\chi\co T(F)\backslash T(\AA)\ra\CC^\times$, up to continuous twists $\abs{\cdot}^{s_1}\otimes\abs{\cdot}^{s_2}$ and twists by the action of $w$, as in Definition~\ref{def:weyl_twist}. We can decompose the category $\Mod^\aut(T)$ into components:
\[
    \Mod^\aut(T)=\prod_{c\in C^\Eis(G)}\Mod^\aut_c(T).
\]
In particular, we have a decomposition of algebras:
\[
    \widehat{i}(\cI_T[\eta^{-1}])=\prod_{c\in C^\Eis(G)}\widehat{i}(\cI_T^c[\eta^{-1}])
\]
as well. Note that as in Remark~\ref{remark:prod_is_coprod}, this product is actually a direct sum.

Recall from Remark~\ref{remark:eis_is_image} that $\cI_G^\Eis$ is isomorphic to the image of the algebra map
\[
    \cI_G\ra\widehat{i}(\cI_T[\eta^{-1}]).
\]
Define the algebras $\cI_G^c$ for $c\in C^\Eis(G)$ via the pullback square:
\[\xymatrix{
    \cI_G^c \ar[r] \ar[d] & \cI_G^\Eis \ar@{^{(}->}[d] \\
    \widehat{i}(\cI_T^c[\eta^{-1}]) \ar@{^{(}->}[r] & \widehat{i}(\cI_T[\eta^{-1}]).
}\]
We claim that:
\begin{proposition} \label{prop:decomp_eis}
    The map:
    \[
        \prod_{c\in C^\Eis(G)}\cI_G^c\ra\cI_G^\Eis
    \]
    is an isomorphism of algebras.
\end{proposition}
\begin{proof}
    Consider the composition
    \[
        \cI_G^\Eis\ra\widehat{i}(\cI_T[\eta^{-1}])=\prod_{c\in C^\Eis(G)}\widehat{i}(\cI_T^c[\eta^{-1}]).
    \]
    It is sufficient to show that for each $c\in C^\Eis(G)$, the projection:
    \[
        \cI_G^\Eis\hookrightarrow\widehat{i}(\cI_T[\eta^{-1}])=\prod_{c'\in C^\Eis(G)}\widehat{i}(\cI_T^{c'}[\eta^{-1}])\twoheadrightarrow\widehat{i}(\cI_T^c[\eta^{-1}])
    \]
    has the same image as its pre-composition with $\cI_G^c\hookrightarrow\cI_G^\Eis$. We will do this by showing that the image of $\cI_G^\Eis$ in $\widehat{i}(\cI_T[\eta^{-1}])$ is closed under the idempotent:
    \[
        e_c\co\widehat{i}(\cI_T[\eta^{-1}])\ra\widehat{i}(\cI_T^c[\eta^{-1}])\ra\widehat{i}(\cI_T[\eta^{-1}]).
    \]
    Indeed, we claim that $e_c$ is given by the action of an element of the center of the category $\Mod(G)$ on $\widehat{i}(\cI_T[\eta^{-1}])$.
    
    This can be shown as follows. Consider the commutative diagram:
    \[\xymatrix{
        \one_T^{\{w_v=1\}} \ar@{^{(}->}[r] \ar[d] & \one_T \ar[d] \\
        \cI_T^{w=1} \ar@{^{(}->}[r] & \cI_T,
    }\]
    Here, $\one_T^{\{w_v=1\}}$ is the sub-algebra of $\one_T=S(T)$ consisting of functions which are separately invariant under the Weyl action of Definition~\ref{def:weyl_twist} at every place. Moreover, $\cI_T^{w=1}$ is the sub-algebra of $\cI_T$ consisting of functions that are invariant under the simultaneous Weyl action at all places. Recall that $\one_T^{\{w_v=1\}}$ is a sub-algebra of the center of $\Mod(G)$.
    
    Since the idempotent $e_c$ is given by averaging against elements of $\cI_T^{w=1}$, it is sufficient to show that the morphism $\xymatrix@1{\one_T^{\{w_v=1\}} \ar[r] & \cI_T^{w=1}}$ is onto. This is the statement of Lemma~\ref{lemma:surj_of_center} below.
\end{proof}

\begin{lemma} \label{lemma:surj_of_center}
    The map
    \[\xymatrix{
        \one_T^{\{w_v=1\}} \ar[r] & \cI_T^{w=1}
    }\]
    from the proof of Proposition~\ref{prop:decomp_eis} is onto.
\end{lemma}

\begin{remark} \label{remark:surj_of_center_explicit}
    To be as explicit as possible, let us give the following equivalent formulation of Lemma~\ref{lemma:surj_of_center}.
    
    Let $f\co T(F)\backslash T(\AA)\ra\CC$ be a smooth and compactly supported function such that it is invariant under the Weyl action:
    \[
        f(a,d)=\abs{\frac{a}{d}}f(d,a).
    \]
    Then there is a smooth and compactly supported function $f'\co T(\AA)\ra\CC$ such that:
    \begin{enumerate}
        \item The function $f'$ lifts $f$:
        \[
            f(a,d)=\sum_{\alpha\in F^\times}\sum_{\delta\in F^\times}f'(\alpha a,\delta d).
        \]
        \item The function $f'$ is locally invariant under the Weyl action at every place $v$ separately:
        \[
            f'(a,d)=\abs{\frac{a_v}{d_v}}_v f'(ad_v/a_v,da_v/d_v),
        \]
        where $a_v$ and $d_v$ denote the components of $a$ and $d$ at $v$.
    \end{enumerate}
\end{remark}

\begin{remark} \label{remark:recompose_characters}
    The point of Lemma~\ref{lemma:surj_of_center} is to show that one cannot ``re-mix'' the local components of two automorphic characters into a different pair of automorphic characters.
    
    Indeed, by going to the spectral picture, Lemma~\ref{lemma:surj_of_center} says the following. Consider a pair of automorphic characters $\chi,\chi'\co T(F)\backslash T(\AA)\ra\CC^\times$, and suppose that by applying the Weyl twist locally at some of the places of $F$, one can turn $\chi$ into $\chi'$. Then either $\chi=\chi'$ or $\chi=\chi'(w)$.
\end{remark}

\begin{proof}[Proof of Lemma~\ref{lemma:surj_of_center}]
    Let us prove the explicit form of Lemma~\ref{lemma:surj_of_center} given by Remark~\ref{remark:surj_of_center_explicit}.
    
    Consider a compact subgroup $K\subseteq\AA^\times$. Let $x,y\in F^\times\backslash\AA^\times/K$ be any two elements. It is enough to show the claim for $f$ supported on the pair of cosets $(x,y)$:
    \[
        f(a,d)=\one_{(xKF^\times,yKF^\times)}(a,d)+\abs{\frac{x}{y}}\one_{(yKF^\times,xKF^\times)}(a,d).
    \]
    
    By the Chebotarev Density Theorem applied to the ray class group
    \[
        F^\times\backslash\AA^\times/K,
    \]
    there is a place $v$ of $F$ such that $\pi_v$ and $xy^{-1}$ are equivalent in the ray class group, where $\pi_v\in F_v^\times$ is a uniformizer. We choose:
    \[
        f'(a,d)=\one_{(\pi_vyK,yK)}(a,d)+\abs{\frac{x}{y}}\one_{(yK,\pi_vyK)}(a,d).
    \]
\end{proof}

Let $C(G)=C^\cusp(G)\coprod C^\Eis(G)$. We summarize our results as follows.
\begin{theorem} \label{thm:decom_to_comps}
    The category $\Mod^\aut(G)$ decomposes into a direct product:
    \[
        \Mod^\aut(G)=\prod_{c\in C(G)}\Mod^\aut_c(G).
    \]
\end{theorem}

\section{Corollaries} \label{sect:results_of_decomp}

In this section, we will give two structural results about the category $\Mod^\aut(G)$, which follow as corollaries of the heavy Theorem~\ref{thm:decom_to_comps}.

The first result, Theorem~\ref{thm:abst_aut_irr_is_aut}, will be that irreducible objects in $\Mod^\aut(G)$ are indeed automorphic in the classical sense. We will prove this in Subsection~\ref{subsect:abst_aut_irr_is_aut}. The second result, Theorem~\ref{thm:fin_len_rest}, will be that automorphic parabolic restriction $\widehat{r}^\aut$ respects finite-length. We will prove this in Subsection~\ref{subsect:fin_len_rest}.

\subsection{Irreducible Abstractly Automorphic Objects} \label{subsect:abst_aut_irr_is_aut}

In this subsection, we will show that all irreducible abstractly automorphic representations are in fact irreducible automorphic representations in the classical sense. This means that they are irreducible sub-quotients of the space of smooth functions on the automorphic quotient $\GL_2(F)\backslash\GL_2(\AA)$.

\begin{theorem} \label{thm:abst_aut_irr_is_aut}
    Every irreducible $V\in\Mod^\aut(G)$ is an irreducible automorphic representation in the classical sense.
\end{theorem}

\begin{proof}
    Let $V$ be as above. We must prove that $V$ is a constituent of the contragradient of the space $S(\GL_2(F)\backslash\GL_2(\AA))$. Note that because $V$ is irreducible, it lies in some component $c\in C(G)$.
    
    If $c$ is Eisenstein, then the action of the center of the category $\Mod^\aut(G)$ on $V$ shows that it is a constituent of a representation of the form $\widehat{i}^\aut(\chi)$, for $\chi\in\Mod^\aut(T)$ irreducible. However, by the main result of \cite{what_is_aut_rep}, this means that $V$ itself is automorphic in the classical sense.
    
    If $c$ is cuspidal, then we note that the action of the center of $G$ shows that $V$ is a module of an irreducible cuspidal automorphic representation $W$ of $G$ with respect to $\oY$. Working place-by-place, the claim follows from Lemma~\ref{lemma:module_is_isom}.
\end{proof}

Recall that in a symmetric monoidal Abelian category $\C$, where the monoidal structure is right-exact, a surjection $\one_\C\twoheadrightarrow A$ determines a unique algebra structure on $A$.

\begin{lemma} \label{lemma:module_is_isom}
    Let $v$ be a place of $F$. Let $W$ be a generic irreducible representation of $\GL_2(F_v)$ and fix a surjection $\one_\Ydown\twoheadrightarrow W$. Let $V\in\Mod(\GL_2(F_v))$ be an irreducible $W$-module with the algebra structure on $W$ induced by $\one_\Ydown\twoheadrightarrow W$. Then $V$ is non-canonically isomorphic to $W$.
\end{lemma}
\begin{proof}
    The claim is easy to see by the action of the center when $W$ is irreducible principle series or supercuspidal. It remains to resolve the case where $W$ is Steinberg. Assume without loss of generality that $W=\St$ is the Steinberg representation with trivial central character, and denote by $\triv_G$ the trivial $G$-module.
    
    Observe that by Example~3.54 of \cite{abst_aut_reps_arxiv}, the unit map $\one_\Ydown\ra\St$ induces an isomorphism:
    \[
        \St\cong\St\oY\St.
    \]
    It follows that any $M\in\GL_2(F_v)$ is an $M$-module if and only if it satisfies
    \[
        M\cong\St\oY M.
    \]
    
    Moreover, one can observe that:
    \begin{align*}
        \St\oY\binom{\triv_G}{\St} & \cong 0 \\
        \St\oY\binom{\St}{\triv_G} & \cong\St,
    \end{align*}
    where $\binom{A}{B}$ denotes the unique non-trivial extension with quotient $A$ and sub-object $B$. This implies that the category of $\St$-modules is projectively generated by the single object $\St$, and thus that every $\St$-module is of the form:
    \[
        \St^{\bigoplus I}
    \]
    for some set $I$.
\end{proof}

\begin{remark}
    Na\"ively, one might have guessed that the category $\Mod^\aut(G)$ is simply given by the subcategory of $\Mod(G)$ where the center of the category acts in a certain way. That is, that an irreducible object $V\in\Mod(G)$ can be tested for abstract automorphicity by only considering the action of the center of $\Mod(G)$ on it. However, note that there are cuspidal automorphic representations which are locally Steinberg in some place $v$. In these cases, replacing the local Steinberg representation at $v$ with the corresponding one-dimensional representation keeps the action of the center the same, but the representation is no longer automorphic. This is the difficulty that Lemma~\ref{lemma:module_is_isom} overcomes.
\end{remark}

\subsection{Finite-Length of Automorphic Parabolic Restriction} \label{subsect:fin_len_rest}

In this subsection, we will show that the functor $\widehat{r}^\aut\co\Mod^\aut(G)\ra\Mod^\aut(T)$ sends finite-length representations to finite-length representations. This will fulfill a debt from Section~\ref{sect:GL2}, where we asserted without proof that $\widehat{r}^\aut$ is well-behaved.

\begin{theorem} \label{thm:fin_len_rest}
    The functor $\widehat{r}^\aut\co\Mod^\aut(G)\ra\Mod^\aut(T)$ sends finite-length representations to finite-length representations.
\end{theorem}

\begin{proof}
    Because $\widehat{r}^\aut$ is right-exact, it is sufficient to show that $\widehat{r}^\aut(V)$ has finite-length for every irreducible $V$.
    
    If $V$ is cuspidal, then $\widehat{r}^\aut(V)$ is $0$ by Corollary~\ref{cor:aut_par_rest_kills_cusp}. It remains to check the case where $V$ belongs to an Eisenstein component. We know that $V$ is a subquotient of some $\widehat{i}^\aut(\chi)$, with $\chi\co T(F)\backslash T(\AA)\ra\CC^\times$. We want to show that the space of co-invariants:
    \[
        \widehat{r}^\aut(V)=\widehat{r}(V)_{/T(F)}
    \]
    is finite-dimensional.
    
    To do this, we need a description of the space $\widehat{r}(V)$. It is given by the restricted product
    \[
        {\bigotimes_v}'\widehat{r}_v(V_v).
    \]
    Each $\widehat{r}_v(V_v)$ is either one of the characters $\chi_v$ or $\chi_v(w_v)$, or it decomposes as:
    \[\xymatrix{
        0 \ar[r] & \chi_v(w_v) \ar[r] & \widehat{r}_v(V_v) \ar[r] & \chi_v \ar[r] & 0,
    }\]
    with $\chi_v(w_v)$ being the Weyl twist of $\chi_v$, as in Definition~\ref{def:weyl_twist}.
    
    Note that there is some additional data needed to understand the restricted product giving $\widehat{r}(V)$, which is the distinguished vector of each local representation $\widehat{r}_v(V_v)$. All we need to know is that at the places where $\widehat{r}_v(V_v)$ is two-dimensional, it is generated as a $T(F_v)$-module by its distinguished vector.
    
    Armed with this description, we can state an informal summary of the idea of the proof. We want to show that $\widehat{r}(V)_{/T(F)}$ is finite-dimensional. Essentially, we want to decompose it into a sum of characters of $T(F)\backslash T(\AA)$. If we were to decompose each object in the restricted tensor product:
    \[
        {\bigotimes_v}'\widehat{r}_v(V_v),
    \]
    then we would see that $\widehat{r}(V)$ is composed only of characters that are either $\chi_v$ or $\chi_v(w_v)$ at every place. However, we want to show that only a finite number of these combinations are \emph{automorphic}, and therefore descend to the quotient $T(F)\backslash T(\AA)$. Indeed, as discussed in Remark~\ref{remark:recompose_characters}, this is the point of Lemma~\ref{lemma:surj_of_center}.
    
    Going back to the formal discussion, we observe that our description of $\widehat{r}(V)$ can be summarised by stating that there is a surjection of $T(\AA)$-modules: 
    \[
        \one_T\otimes_{\one_T^{\{w_v=1\}}}\left(\chi\middle|_{\one_T^{\{w_v=1\}}}\right)\twoheadrightarrow\widehat{r}(V)
    \]
    where $\left(\chi\middle|_{\one_T^{\{w_v=1\}}}\right)$ is the restriction of $\chi\co\one_T\ra\CC$ along $\one_T^{\{w_v=1\}}\ra\one_T$. Here, as in Lemma~\ref{lemma:surj_of_center}, the space $\one_T^{\{w_v=1\}}$ is the space of smooth and compactly supported functions on $T$ that are invariant under the Weyl action $w_v$ at every place $v$ of $F$. This means that:
    \[
        \widehat{r}(V)_{/T(F)}
    \]
    is covered by:
    \[
        \cI_T\otimes_{\one_T^{\{w_v=1\}}}\left(\chi\middle|_{\one_T^{\{w_v=1\}}}\right)\cong\cI_T\otimes_{\cI_T^{w=1}}\left(\chi\middle|_{\cI_T^{w=1}}\right),
    \]
    where the isomorphism follows via Lemma~\ref{lemma:surj_of_center}. However, $\cI_T$ is free of rank $2$ over $\cI_T^{w=1}$, meaning that $\widehat{r}(V)_{/T(F)}$ is at most two-dimensional.
\end{proof}

\begin{remark}
    In the course of the proof of Theorem~\ref{thm:fin_len_rest}, we have in fact shown that $\widehat{r}^\aut(V)$ is covered by an extension
    \[
        \binom{\chi}{\chi(w)}
    \]
    of some character $\chi\co T(F)\backslash T(\AA)\ra\CC^\times$ by its twist $\chi(w)$.
\end{remark}

\Urlmuskip=0mu plus 1mu\relax
\bibliographystyle{alphaurl}
\bibliography{L_func}

\end{document}